\documentclass[a4paper, 12pt]{article}
\usepackage{amsmath,amssymb,amsthm}
\usepackage{graphicx,color}
\usepackage{hyperref}
\usepackage{float}
\usepackage{capt-of}
\usepackage{etex}
\usepackage{mathtools}
\usepackage{stmaryrd}
\usepackage{authblk}
\usepackage{epstopdf}
 
\newcommand{\BB}{\mathbb{M}}
\newcommand{\QQ}{\mathbb{Q}}
\newcommand{\PP}{\mathbb{P}}

\newcommand{\ve}{{\varepsilon}}

\newcommand{\ben}{\begin{equation}}
\newcommand{\een}{\end{equation}}
\newcommand{\benn}{\begin{equation*}}
\newcommand{\eenn}{\end{equation*}}

\newcommand{\R}{\mathbf{R}}
\newcommand{\N}{\mathbf{N}}

\newcommand{\divv}{\operatorname{div}}

\newcommand{\Ltwo}{L^2}

\newcommand{\dt}{{\frac{d}{dt}}}

\newcommand{\e}{\varepsilon}

\newcommand{\ap}{\mathcal{A}}

\newcommand{\oms}{{\Omega^\star}}
\newcommand{\us}{u^\star}
\newcommand{\bonen}{\beta_1^n}
\newcommand{\btwon}{\beta_2^n}
\newcommand{\bonem}{\beta_1^m}
\newcommand{\btwom}{\beta_2^m}
\newcommand{\Omr}{\Omega^{\text{ref}}} 
\newcommand{\Oma}{\Omega_{\text{ma}}}

\newtheorem{theorem}{Theorem}
\newtheorem{lemma}{Lemma}

\newtheorem{remark}{Remark}

\newtheorem{definition}{Definition}

\newtheorem{assumption}{Assumption}
\newtheorem{algorithm}{Algorithm}
\newtheorem{proposition}{Proposition}

\begin{document}

\title{Shape optimization of an electric motor subject to nonlinear magnetostatics}

\author[1]{P. Gangl \thanks{peter.gangl@dk-compmath.jku.at}}
\author[1,2]{U. Langer \thanks{ulrich.langer@ricam.oeaw.ac.at}}
\author[3]{A. Laurain \thanks{laurain@math.tu-berlin.de}}
\author[3]{H. Meftahi \thanks{meftahi@math.tu-berlin.de}}
\author[4]{K. Sturm \thanks{ kevin.sturm@uni-due.de}}
\affil[1]{Johannes Kepler University, Altenberger Stra\ss e 69, A-4040 Linz, Austria}
\affil[2]{RICAM, Altenberger Stra\ss e 69, A-4040 Linz, Austria}
\affil[3]{Technical University of Berlin, Str. des 17. Juni 136, 10623 Berlin}
\affil[4]{Universit\"at Duisburg-Essen, Fakult\"at f\"ur Mathematik}
\providecommand{\keywords}[1]{\textbf{\textbf{Keywords:}} #1}

\date{}
\maketitle

\begin{abstract}
The goal of this paper is to improve the performance of an electric motor by modifying the geometry of a specific part of the iron core of its rotor. To be more precise, the objective is to smooth the rotation pattern of the rotor. A shape optimization problem is formulated by introducing a tracking-type cost functional to match a desired rotation pattern. The magnetic field generated by permanent magnets is modeled by a nonlinear partial differential equation of magnetostatics. The shape sensitivity analysis is rigorously performed for the nonlinear problem by means of a new shape-Lagrangian formulation adapted to nonlinear problems. 
\end{abstract}
\section{Introduction}

Advanced shape optimization techniques have become a key tool for the design of industrial structures. In the automotive and aeronautic industries, for instance, the reduction of the drag or of the noise are important features which can be reduced by changing the design of the vehicles. In general, when considering a complex mechanical assemblage, it is often possible to optimize the geometry of certain pieces to improve the overall performance of the object. In the industrial sector, the shape optimization of electrical machines is the most economical approach to improve their efficiency and performance.
Shape optimization problems are usually formulated as the minimization of a given cost function, typical examples being the weight or the compliance for elastic systems. The most interesting and challenging problems of these type have linear or nonlinear partial differential equations constraints; see, for instance, \cite{AmLa,DZ2, MR2523581, HaslingerMaekinen2003, HaslingerNeittaanmaeki1996, MR2653723, MR2806573, ik, LaporteLeTallec2003, COV:9228529,MR3011966,LaurainSturm,o,dlr86439,SZ} and the references therein.

In this work, a shape optimization approach is used to improve the design of an electric motor in order to match a desired smoother rotation pattern. As a model problem, we consider an interior permanent magnet (IPM) brushless electric motor consisting of a rotor (inner part) and a stator (outer part) separated by a thin air gap and containing both an iron core; see Fig. \ref{fig1} for a description of the geometry. The rotor contains permanent magnets. The coil areas are located on the inner part of the stator. In general, inducing current in the coils initiates a movement of the rotor due to the interaction between
the magnetic fields generated by the electric current and by the magnets. In our application, we are only interested in the magnetic field $B$ for a fixed rotor position without any current induced. Since the magnetic properties of copper in the coils and of air are similar, we model these regions as air regions, too. We refer the reader to \cite{ChoiEtAl2011, HongKwackMin2010, MiyagiEtAl2011, MiyagiTakahashiYamada2010} for other approaches to the design optimization of IPM electric motors and to \cite{6556323} for a special method of modeling IPM motors using radial basis functions.

Due to practical restrictions, only some specific parts of the geometry can be modified. In our application, we identify a design subregion $\Omega$ of the iron core of the rotor subject to the shape optimization process. Our objective is to modify $\Omega$ in order to match a desired rotation pattern as well as possible. Practically this is achieved by tracking a certain desired profile of the magnetic flux density, which is done by reformulating the problem as a shape optimization one by introducing a tracking-type cost function.

The shape optimization of $\Omega$ has been considered in \cite{GanglLanger2014} from the point of view of the topological sensitivity \cite{MR3013681,MR1691940}. However, the derivation of the so-called {\it topological derivative} for nonlinear problems is formal since the mathematical theory for these problems is still in its early stages; 
see \cite{Alain2013,MR2356899,MR2541192} for a few results in this direction. Moreover, the drawback of the topological derivative is that it usually creates geometries with jagged contours.

In this paper, we focus on the shape optimization of the design domain $\Omega$ by means of the {\it shape derivative}, which, contrarily to the topological derivative, proceeds by smooth deformation of the boundary of a reference design. In this way, the optimal shape has a smooth boundary provided that the numerical algorithm is carefully devised. Computing the shape derivative for problems depending on linear partial differential equations is a well-understood topic; see for instance \cite{DZ2,MR2512810,SZ}. For nonlinear problems, the literature is scarcer and the computation of the shape derivative is often formal. A novel aspect of this paper is to provide an efficient and rigorous way to compute the shape derivative of the cost functional without the need to compute the material derivative of the solution of the nonlinear state equation. 
The method is based on a novel Lagrangian method for nonlinear problems and on the volume expression of 
the shape derivative; see \cite{LaurainSturm,sturm}. 
This allows to obtain a smooth deformation field used as a descent direction in a gradient method. In the numerical algorithm, the mesh is deformed iteratively using this vector field until it reaches an equilibrium state.

The rest of the paper is organized as follows: In Section \ref{sec:problemFormulation}, we formulate the shape optimization problem and give the underlying nonlinear magnetostatic equation. Existence of a solution to the shape optimization problem is shown in Section \ref{sec:existence}. 
In Section \ref{sec:shapDerivative}, we introduce the general notion of a shape derivative and give an abstract differentiability result which is used later on to compute the shape derivative of the cost functional. 
Section \ref{sec:shapeDerivativeCostFunc} deals with the shape derivative of the cost functional. Finally, in Section \ref{sec:numerics}, a numerical algorithm is presented to optimize the design of $\Omega$, and numerical results showing the optimal shape are presented.

%---------------------------------------------------------------------------%
 \begin{figure}
 \begin{minipage}{7cm}
 \begin{tabular}{c}
  \hspace{-5mm}\includegraphics[scale=0.45, trim=300 0 300 0, clip=true]{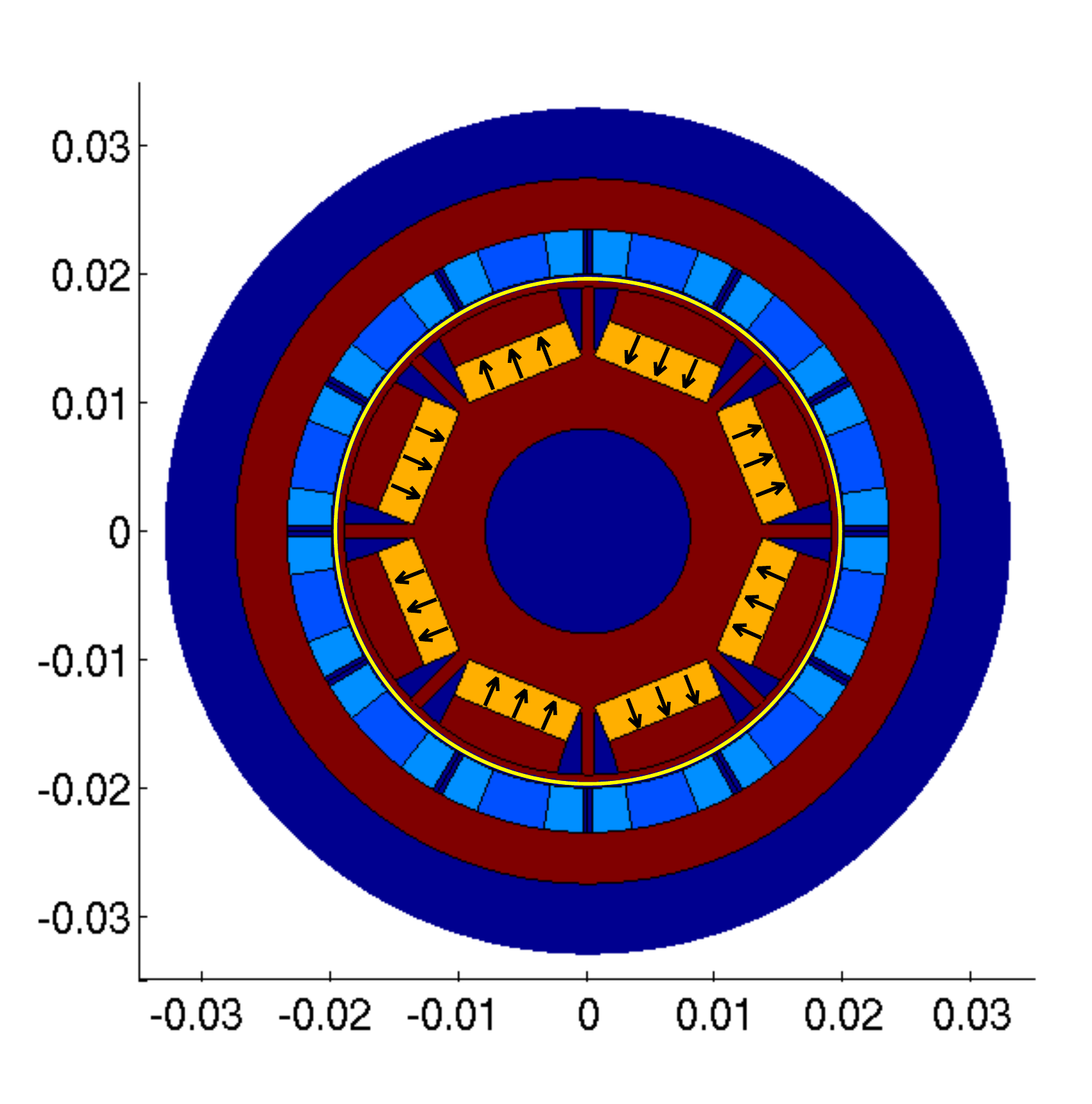} \\
    (a)
  \end{tabular} \end{minipage}
    \hspace{1.2cm}
   \begin{minipage}{7cm}
   \begin{tabular}{c}
		\hspace{-10mm}\includegraphics[scale=0.45, trim=300 0 300 0, clip=true]{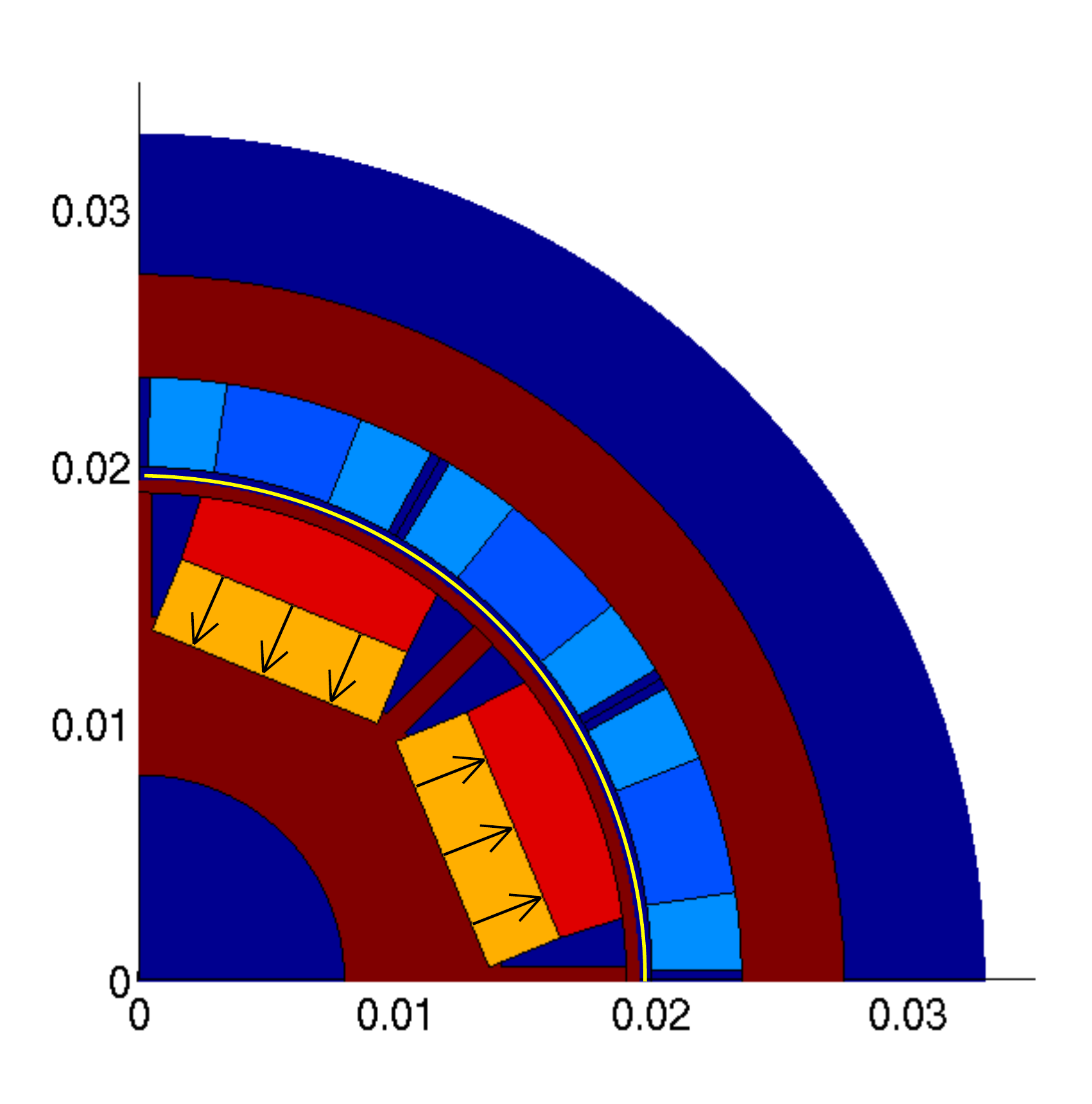}\\
       \hspace{-10mm}(b)
\end{tabular}
 \end{minipage}
 \caption{(a) Geometry of the electric motor with permanent magnet areas $\Omega_{\text{ma}}$ (orange with arrows indicating the directions of magnetization), ferromagnetic material $\Omega_{f}^{\text{ref}}$ (brown), coil areas $\Omega_c$ (light blue), air $\Omega_a$ (dark blue) and thin air gap (yellow layer between rotor and stator); $\Gamma_0$ is a circle located in the air gap. (b) Zoom on the upper right quarter of the motor where the reference design area $\Omr \subset \Omega_f^{\text{ref}}$ corresponds to the highlighted (red) region. The reference region $\Omega^{\text{ref}}$ is subject to the shape optimization procedure. 
}
  \label{fig1}
  \end{figure}
%--------------------------------------------------------------------------%

\section{Problem formulation}    \label{sec:problemFormulation}
Let $D\subset\R^2$ be the smooth bounded domain representing an IPM brushless electric motor as depicted in Fig. \ref{fig1} with a ferromagnetic part $\Omega_{f}^{\text{ref}}\subset D$, permanent magnets $\Omega_{\text{ma}}\subset D$, air regions $\Omega_a\subset D$ and coil areas $\Omega_c$.
The design domain $\Omega$ is included in a reference domain $\Omega^{\text{ref}}\subset \Omega_{f}^{\text{ref}}$. The inner part of the motor is called the rotor and the outer part the stator. They are separated by a small air gap, the thin yellow circular layer in Fig. \ref{fig1}. By $\Gamma_0$ we denote a circle within the air gap.
When an electric current is induced in the coils, the rotor containing the permanent magnets rotates.
In reality the motor is a three-dimensional object, but considering the problem only on the cross-section of the motor is a modeling assumption that is commonly used; see \cite{ArumugamEtAl1985, binder2012elektrische}. For a comparison between two- and three-dimensional models of electric motors, see \cite{KolotaSteien2011, TorkamanAfjei2008}.

Denote $\Gamma := \partial \Omega$ the boundary of the optimized part $\Omega$ which is assumed to be Lipschitz. We introduce the variable ferromagnetic set $\Omega_f := (\Omega_f^{\text{ref}} \setminus \overline{\Omega^{\text{ref}}})\cup \Omega$ and $\Gamma_f: = \partial \Omega_f$. The permanent magnets create a magnetic field in $D$. In our application, we assume the coils to be switched off. Thus, no electric current is induced and the rotor is not moving. The magnetic field generated by the permanent magnets can be calculated via a boundary value problem of the form
\begin{equation}
 \label{state}
    \begin{aligned}
        -\divv(\beta_\Omega(x,|\nabla u|^2)\nabla u)  &= f \quad \text{ in } \Omega_f  \text{ and } D\setminus\overline{\Omega}_f, \\
                                       u&=0\quad \text{ on } \partial D,
    \end{aligned}
\end{equation}
with the transmission conditions on the interface $\Gamma_{f}$
\begin{equation}
\label{trans_cond_state}
\begin{aligned}
             \llbracket u \rrbracket  = 0 & \quad \text{ on } \quad \Gamma_f, \\
         \llbracket  \beta_\Omega(x,|\nabla u|^2)\partial_n u\rrbracket  = 0 & \quad \text{ on } \quad \Gamma_f,
\end{aligned}
\end{equation}
where $n$ denotes the outward unit normal vector to $\Omega_f$. Defining $v^-$ the restriction of some function $v$ on $\Omega_f$ 
and $v^+$ its restriction on $D\setminus\overline{\Omega_f}$ we denote by $\llbracket v \rrbracket$ the jump of $v$ across the interface $\Gamma_f$, i.e. 
\[
\llbracket v\rrbracket = v^+|_{\Gamma_f} - v^-|_{\Gamma_f}.
\]
The nonlinear, piecewise smooth function $\beta$ is defined for $(x,\zeta)\in D\times\R$ as
\begin{align*}
\beta_\Omega(x,\zeta)  
&:=\beta_{1}(\zeta)  \chi_{\Omega_f}(x) + \beta_{2}(\zeta)  \chi_{D\setminus \overline{\Omega_{f}}}(x)\\ 
 &=  \beta_{1}(\zeta)  (\chi_{\Omega}(x) + \chi_{\Omega_{f}^{\text{ref}}\setminus \overline{\Omr}}(x)   ) + \beta_{2}(\zeta)  ( \chi_{D\setminus \overline{\Omega_{f}^{\text{ref}}}}(x)  +  \chi_{\Omr\setminus \overline{\Omega}}(x)),
\end{align*}
where $\chi$ is the indicator function of a given set. Note that the expression above is meaningful since $\Omega\subset \Omr\subset \Omega_{f}^{\text{ref}}$. The weak form of \eqref{state} reads
\begin{align} \label{weak_state}
    \mbox{Find } u \in H^1_0(D)\mbox{ such that } &\int_D \beta_\Omega(x,|\nabla u|^2) \nabla u \cdot \nabla v \, \mbox d x = \langle f,v\rangle\quad \forall v\in H^1_{0}(D),
\end{align}
where $\langle\cdot ,\cdot \rangle$ denotes the duality bracket between $H^{-1}(D)$ and $H^{1}_0(D)$. The scalar function $u$ is the third component of the vector potential of the magnetic flux density in three dimensions, $B = \mbox{curl} ( (0, 0, u)^T)$. In our model, we consider the restriction of $B$ to a two-dimensional cross-section since the third component vanishes. 

In the sequel, we make the following assumption for $\beta_1$ and $\beta_2$: 
\begin{assumption} \label{assump:beta}
The functions $\beta_{1}, \beta_{2} : \R \rightarrow \R$ satisfy the following conditions:
	\begin{enumerate}
		\item \label{assump:beta_1}There exist constants $c_{1}, c_{2},c_{3}, c_{4}>0$, such that 
			\[
				c_{1}\leq\beta_{1}(\zeta)\leq c_{2}, \quad      c_{3}\leq\beta_{2}(\zeta)\leq c_{4} \quad \text{ for all }  \zeta\in \R. 
			\] 
		\item \label{assump:beta_2} The function $s \mapsto \beta_i(s^2)s$ is strongly monotone with monotonicity constant $m$ and Lipschitz continuous 
		with Lipschitz constant $L$:
	\begin{align*}
		(\beta_i(s^2)\, s - \beta_i(t^2)\,t)\, (s-t) &\geq m (s-t)^2 \quad &&\text{ for all } s, t \geq 0, \\
		|\beta_i(s^2)\,s - \beta_i(t^2)\,t)| &\leq L  |s-t| \quad &&\text{ for all } s, t \geq 0.
	\end{align*}
		\item \label{assump:beta_3} The functions $\beta_{1}, \beta_{2}$ are in $C^1(\R)$. 
		\item \label{assump:beta_4}There exist constants $\lambda, \Lambda>0$ such that for $i=1,2$,
			\[
				\lambda | \eta |^2\leq \beta_{i}( |\rho |^2) | \eta |^2+ 2\beta'_{i}(|\rho |^2)(\eta \cdot\rho)^2\leq \Lambda | \eta |^2 \quad \text{ for all } \eta, \rho \in \R^2.
			\]
	\end{enumerate}
\end{assumption}
The task is to modify the shape of the design region $\Omega\subset \Omega^{\text{ref}}$ in such a way that the radial component of the resulting magnetic flux density on the circle $\Gamma_0$ in the air gap fits a given data as good as possible. We consider the following minimization problem:
%%%%%%%%%%%%%%%%%%%%%%%%%%%%
%%%%%%%%%%%%%%%%%%%%%%%%%%%%
\begin{equation}\label{opt}
\left\{\begin{aligned}
&\mbox{minimize } J(\Omega,u) :=\int_{\Gamma_0} |B_r(u)-B_d|^2 \mbox ds\\
&\text{subject to } \Omega \, \in \, \mathcal{O} \text{ and } u \text{ solution of \eqref{weak_state} } 
\end{aligned}\right.
\end{equation}
%%%%%%%%%%%%%%%%%%%%%%%%%%%%
%%%%%%%%%%%%%%%%%%%%%%%%%%%%
where 
\begin{equation}\label{admissible}
\mathcal{O}= \left\{\Omega \subset \Omega^{\text{ref}}\subset \Omega_{f}^{\text{ref}}, \Omega \text{ open and Lipschitz with uniform Lipschitz constant }L_{\mathcal O} \right\}
\end{equation}
and $\Gamma_{0}\subset D$ is a smooth one-dimensional subset of $D$. The sets $\Omega^{\text{ref}}$ and $\Omega^{\text{ref}}_f$ are reference domains; see Fig. \ref{fig1}. Here, $B_r(u) = \nabla u\cdot\tau $ where $\tau$ is the tangential vector to $\Gamma_{0}$ and $B_d \in C^1(\Gamma_0)$ denotes the given desired radial component of the magnetic flux density along the air gap. 
In order to obtain the first-order optimality conditions for this minimization problem we compute the derivative of $J$ with respect to the shape $\Omega$. 
\begin{remark} Let us note that
	\begin{itemize}
		\item in our application, the right-hand side $f$ represents the magnetization of the permanent magnets. In general, it can be a combination of magnetization and impressed currents in the coils. 
		\item the link between $\beta$ and the magnetic reluctivity is $\beta(s) = \nu(\sqrt{s})$. In this case, the boundary value problem \eqref{state} is called the two-dimensional magnetostatic boundary value problem; see \cite{JungLanger:2013a, Pechstein2008a}. We will see in Section \ref{sec:numerics} that Assumption \ref{assump:beta} is satisfied for $\nu$ due to physical properties.
		\item the Dirichlet condition $u|_{\partial D} = 0$ implies $B \cdot n = 0$, thus no magnetic flux can leave the computational domain $D$.
	\end{itemize}
\end{remark}
Given Assumption \ref{assump:beta}, we can state existence and uniqueness of a solution $u$ to the state equation \eqref{weak_state}. 
\begin{theorem} \label{thm:exState}
Assume that Assumption \ref{assump:beta}.\ref{assump:beta_1} and \ref{assump:beta}.\ref{assump:beta_2} hold. Then problem \eqref{weak_state} admits a unique solution $u \in H^1_0(D)$ for any fixed right-hand side $f \in H^{-1}(D)$ and we have the estimate
$$\|u\|_{H_0^1(D)} \leq c \| f \|_{H^{-1}(D)}.$$
\end{theorem}
%%%%%%%%%%%%
\begin{proof}
A proof can be found in \cite{Heise1994, Pechstein2004, Zeidler1995}.
\end{proof}
%%%%%%%%%%%%%%%%%%%%%%%%%%%
\section{Existence of optimal shapes} \label{sec:existence}
%%%%%%%%%%%%%%%%%%%%%%%%%%%

In this section, we prove that problem \eqref{opt} has a solution $\oms$. We make use of the following result \cite[Theorem 2.4.10]{MR2512810}
\begin{theorem}\label{lipCV}
Let $\Omega_n$ be a sequence in $\mathcal{O}$. Then there exists $\Omega\in\mathcal{O}$ and a subsequence $\Omega_{n_k}$ which converges to $\Omega$ in the sense of Hausdorff, and in the sense of characteristic functions. In addition, $\overline{\Omega}_{n_k}$ and $\partial\Omega_{n_k}$ converge in the sense of Hausdorff respectively towards $\overline{\Omega}$ and $\partial\Omega$. 
\end{theorem}
Let $\Omega_n\in\mathcal{O}$ be a minimizing sequence for problem \eqref{opt}. According to Theorem \ref{lipCV}, we can extract a subsequence, which we still denote $\Omega_n$, which converges to some $\Omega\in\mathcal{O}$. Denote $u_n$ and $u$ the solutions of \eqref{state}-\eqref{trans_cond_state} with $\Omega_n$ and $\Omega$, respectively. We prove that $u_n\to u$ in $H^1_0(D)$.

First of all in view of Theorem \ref{thm:exState} we have
\begin{align} \label{eq:un_bounded}
	\| u_n \|_{H^1_0(D)}  \leq c \| f \|_{H^{-1}(D)}
\end{align}
where $c$ depends only on $D$. Thus we may extract a subsequence $u_n$ such that $ u_n \to \us$ in $\Ltwo(D)$ and $\nabla u_n \rightharpoonup \nabla\us$ weakly in $\Ltwo(\Omega)$. Extracting yet another subsequence, we may as well assume that $\Omega_n \to \oms$ in the sense of characteristic functions applying  Theorem~\ref{lipCV}.

Subtracting the variational formulation for two elements $u_n$ and $u_m$ of the sequence and choosing the test function $v= u_n - u_m$ we get
\begin{align} \label{6}
&\int_D \left(  (\beta_{\Omega_n}(x,|\nabla u_n|^2) \nabla u_n    -   (\beta_{\Omega_m}(x,|\nabla u_m|^2) \nabla u_m  \right) \cdot \nabla (u_n - u_m) \, \mbox d x = 0
\end{align}
Let us introduce for simplicity the notation $\bonen := \beta_1(|\nabla u_n|^2)$ and $\btwon := \beta_2(|\nabla u_n|^2)$. Then \eqref{6} becomes
\begin{align*} 
&\int_D  (     \bonen    (\chi_{\Omega_n}               + \chi_{\Omega_{f}^{\text{ref}}\setminus \Omr}   )         
             + \btwon    (\chi_{D\setminus \Omega_{f}^{\text{ref}}}  + \chi_{\Omr\setminus \Omega_n}     )) \nabla u_n \cdot \nabla (u_n - u_m)\\
& -   (        \bonem    (\chi_{\Omega_m}               + \chi_{\Omega_{f}^{\text{ref}}\setminus \Omr}   )         
             + \btwom    (\chi_{D\setminus \Omega_{f}^{\text{ref}}}  + \chi_{\Omr\setminus \Omega_m}     )) \nabla u_m \cdot \nabla (u_n - u_m) \, \mbox d x = 0
\end{align*}
This yields 
\begin{equation}\label{Is=0}
I_1 + I_2 + I_3 + I_4 + I_5 + I_6 = 0 
\end{equation}
where
\begin{align*} 
I_1 := & \int_D (\chi_{\Omega_n} - \chi_{\Omega_m})  \bonen \nabla u_n     \cdot \nabla (u_n - u_m)\,dx  ,\\
I_2 := & \int_D  \chi_{\Omega_m}(\bonen \nabla u_n   - \bonem\nabla u_m   )\cdot \nabla (u_n - u_m) \,dx ,\\
I_3 := & \int_D  \chi_{\Omega_{f}^{\text{ref}}\setminus \Omr} (\bonen  \nabla u_n - \bonem\nabla u_m) \cdot \nabla (u_n - u_m) \,dx,\\
I_4 := & \int_D  \chi_{D\setminus \Omega_{f}^{\text{ref}}}    (\btwon  \nabla u_n - \btwom\nabla u_m) \cdot \nabla (u_n - u_m) \,dx,\\
I_5 := & \int_D (\chi_{\Omr\setminus \Omega_n} - \chi_{\Omr\setminus \Omega_m})  \btwon \nabla u_n     \cdot \nabla (u_n - u_m)\,dx , \\
I_6 := & \int_D  \chi_{\Omr\setminus \Omega_m}(\btwon \nabla u_n   - \btwom\nabla u_m   )\cdot \nabla (u_n - u_m) \,dx .
\end{align*}
To estimate the above integrals, we use the following lemma:
%%%%%%%%%%%%%%%%%%%%%%%%%%%%%
\begin{lemma}\label{lemma1}
Let $p,q\in\R^2$. If Assumption \ref{assump:beta}.\ref{assump:beta_2} holds, then
$$ [\beta_i(|p|^2)p - \beta_i(|q|^2)q ]\cdot(p-q)  \geq m |p-q|^2. $$
\end{lemma}
%%%%%%%%%%%%%%%%%%%%%%%%%%%%
\begin{proof}
Using Assumption \ref{assump:beta}.\ref{assump:beta_2} we compute
\begin{align*}
[\beta_i(|p|^2)p & - \beta_i(|q|^2)q ]  \cdot(p-q) \\
& =  m |p-q|^2  + [(\beta_i(|p|^2) - m)p - (\beta_i(|q|^2) - m)q ]\cdot(p-q) \\
& =  m |p-q|^2  + \underbrace{(\beta_i(|p|^2) - m)}_{\geq 0}  \underbrace{(|p|^2 - p\cdot q)}_{\geq  |p|(|p| - |q|)} -  \underbrace{(\beta_i(|q|^2) - m)}_{\geq 0} \underbrace{(p\cdot q  -|q|^2)}_{\leq |q|(|p| - |q|)}\\
& =  m |p-q|^2  + \underbrace{(\beta_i(|p|^2)|p|  - \beta_i(|q|^2)|q|)  (|p| - |q|)}_{\geq m (|p| - |q|)^2} - m(|p| - |q|)^2\\
&  \geq m |p - q|^2.
\end{align*}
\end{proof}
Applying Lemma~\ref{lemma1} with $p = \nabla u_n$ and $q = \nabla u_m$ we get
\begin{align} 
\notag I_2 +I_3 + I_4 + I_6 &\geq   m \int_D  (\chi_{\Omega_m}  +  \chi_{\Omega_{f}^{\text{ref}}\setminus \Omr}  +  \chi_{D\setminus \Omega_{f}^{\text{ref}}}  +  \chi_{\Omr\setminus \Omega_m})   |\nabla u_n  -  \nabla u_m|^2\,dx \\
\label{Is2} &  =   m \int_D  |\nabla u_n  -  \nabla u_m|^2\,dx.
\end{align}
H\"older's inequality yields
$$|I_1| \leq \| \bonen \|_{L^{\infty}(D)} \| \chi_{\Omega_n} - \chi_{\Omega_m}  \|_{L^r(D)}  \| \nabla u_n \|_{\Ltwo(D)}   \| \nabla (u_n - u_m) \|_{L^{s}(D)} $$
with $r^{-1} + s^{-1} = 1/2$, $r,s\geq 1$. Performing a similar estimate for $I_5$ and in view of Assumption \ref{assump:beta}.\ref{assump:beta_1} this yields
\begin{align} 
\label{Is3} |I_1| &  \leq c_2\| \chi_{\Omega_n} - \chi_{\Omega_m}  \|_{L^r(D)}  \| \nabla u_n \|_{\Ltwo(D)}   \| \nabla (u_n - u_m) \|_{L^{s}(D)} \\
\label{Is4} |I_5| &  \leq c_4\| \chi_{\Omega_n} - \chi_{\Omega_m}  \|_{L^r(D)}  \| \nabla u_n \|_{\Ltwo(D)}   \| \nabla (u_n - u_m) \|_{L^{s}(D)}  
\end{align}
Using equality \eqref{Is=0} and inequalities \eqref{eq:un_bounded},\eqref{Is2},\eqref{Is3},\eqref{Is4} we obtain the estimate
\begin{align*} 
\int_D  |\nabla u_n  -  \nabla u_m|^2\,dx \leq c\| \chi_{\Omega_n} - \chi_{\Omega_m}  \|_{L^r(D)} \|f\|_{H^{-1}(D)}   \| \nabla (u_n - u_m) \|_{L^{s}(D)} 
\end{align*}
Since $\chi_{\Omega_n}$ is a characteristic function, the parameter $r$ can be chosen arbitrarily large, and consequently, $s$ can be chosen arbitrarily close to $2$. Therefore, assuming a little more regularity than $H^1$ for the solution of \eqref{weak_state}, the convergence of the characteristic functions of $\Omega_n$ in $L^p(D)$ implies the strong convergence of $u_n$ towards $\us$ in $H^1_0(D)$. Consequently, we obtain the following result:
\begin{proposition}
Let $\Omega_n\in\mathcal{O}$ be a minimizing sequence for problem \eqref{opt} and $\Omega$ be an accumulation point of this sequence as in Theorem \ref{lipCV}. Assume there exists $\e>0$ such that the solution $u$ of \eqref{weak_state} satisfies
$$ \| u \|_{H^{1+\e}(D)\cap H^1_0(D)}  \leq c  $$
where $c$ depends only on $f$ and $D$. Then the sequence $u_n\in H^1_0(D)$ corresponding to $\Omega_n$ converges to $u$ strongly in $H^1_0(D)$, where $u$ is the solution of \eqref{weak_state} in $\Omega$.
\end{proposition}
\begin{proof}
We have seen that there exists $\us\in H^1_0(D)$ such that $u_n\to \us$ in $H^1_0(D)$. We just need to prove that $\us = u$. The strong convergence of $u_n$ in $H^1_0(D)$ yields $\nabla u_n \to\nabla \us$ pointwise almost everywhere in $D$, and also the pointwise almost everywhere (a.e.) convergence $\beta_i^n \to \beta_i(|\nabla \us|^2)$ for $i=1,2$. We also have the pointwise a.e. convergence of the characteristic function $\chi_{\Omega_n}$ to $\chi_{\Omega^\star}$ which implies the pointwise a.e. convergence $\beta_{\Omega_n}(\cdot, |\nabla u_n|^2) \to \beta_{\Omega^\star}(\cdot,|\nabla \us |^2)$. Next, the weak formulation for $u_n$ is 
\begin{align*} 
\int_D  \beta_{\Omega_n}(x,|\nabla u_n|^2)\nabla u_n \cdot \nabla v \, \mbox d x = \langle f,v\rangle\quad \text{ for all } v\in H^1_{0}(D),
\end{align*}
The strong convergence of $u_n$ in $H^1_0(D)$ and the pointwise convergence of $\beta_{\Omega_n}(\cdot ,|\nabla u_n|^2)$ implies
\begin{align*} 
\int_D  \beta_{\oms}(x,|\nabla \us|^2)\nabla \us \cdot \nabla v \, \mbox d x = \langle f,v\rangle\quad \text{ for all } v\in H^1_{0}(D),
\end{align*}
which proves finally $\us = u$.
\end{proof}
%%%%%%%%%%%%%%%%%%%%%%%%%%%%%%%
\begin{remark}\label{rem:lipschitz}
In fact we have proven a stronger result, i.e., the Lipschitz continuity of $u_n$ in $H^1_0(D)$ with respect to the characteristic function $\chi_{\Omega_n}$. 
\end{remark}

%%%%%%%%%%%%%%%%%%%%%%%%%%%
\section{Shape derivative} \label{sec:shapDerivative}
%%%%%%%%%%%%%%%%%%%%%%%%%%%
\subsection{Preliminaries} 
%%%%%%%%%%%%%%%%%%%%%%%%%%%%
In this section, we recall some basic facts about the velocity method in shape optimization used to transform a reference shape; see \cite{DZ2,SZ}. In the velocity method, also known as speed method, a domain $\Omega\subset D\subset \R^2$ is deformed by the action of a velocity field $V$ defined on $D$. Suppose that $D$ is a Lipschitz domain and denote its boundary $\Sigma:=\partial D$. The domain evolution is described by the solution of the dynamical system
 \begin{equation}
 \label{speed}
 \frac{d}{dt}x(t)=V(x(t)),  t\in [0,\varepsilon ), \quad x(0)=X\in \R^2
\end{equation} 
for some real number $\varepsilon>0$.  
Suppose that $V$ is continuously differentiable and has compact support in $D$, i.e. $V\in \mathcal{D}^1(D,\R^2)$. 
Then the ordinary differential equation \eqref{speed} has a unique solution on $[0,\varepsilon)$. 
This enables us to define the diffeomorphism
 \begin{equation}
 \label{transf}
 T_{t} :\R^2\rightarrow \R^2;  X\mapsto T_{t}(X):=x(t). 
 \end{equation}
With this choice of $V$, the domain $D$ is globally invariant 
by the transformation $T_t$, i.e. $T_t(D) = D$ and $T_t(\partial D) = \partial D$.  
For $t\in [0,\varepsilon)$, $ T_t$ is invertible. Furthermore, for 
sufficiently small $\tau>0$, the 
Jacobian determinant 
\begin{equation}\label{JacobiDet}
	\xi(t):=\text{det}DT_{t}
\end{equation}
of $T_t$ is strictly positive. 
In the sequel, we use the notation $DT_{t}^{-1}$ for the inverse of $DT_{t}$ and $DT_{t }^{-T}$
for the transpose of the inverse. We also denote by 
\begin{equation} \label{tangJacobiDet}
	\xi_\tau(t):= \xi(t)|DT_{t}^{-T}n|
\end{equation}
the tangential Jacobian of $T_{t}$ on $\partial  D$.\\
Then the following lemma holds \cite{DZ2}:
\begin{lemma}
For \(\varphi \in W^{1,1}_{loc}(\R^2) $ and $ V \in \mathcal{D}^1(\R^2,\R^2)$ we have 
\begin{align*}
\nabla (\varphi \circ T_{t})	&=DT_{t}^T(\nabla \varphi)\circ T_{t}, \qquad\qquad  \frac{d}{dt}(\varphi \circ T_{t})	=(\nabla \varphi \cdot V)\circ T_{t}, \\
\frac{d\xi(t)}{dt}    &=  \xi(t) \left[  \divv V(t)\right]\circ T_{t}, \qquad\qquad \xi_\tau'(0) = \divv_{\partial D}V := \divv V|_{\partial D}-DV n\cdot n.
\end{align*}
\end{lemma}
%%%%%%%%%%%%%%
\begin{definition}
\label{def1}
Suppose we are given a real valued shape function $J$ defined on a subset $\Xi$ 
of the powerset $2^{\R^2}$.
We say that $J$ is {\it Eulerian semi-differentiable} at $\Omega\in \Xi$ in the direction $V$ if the following limit exists in $\R$
\[
dJ(\Omega;V) := \lim_{t\searrow 0}\frac{J(T_{t}(\Omega))-J(\Omega)}{t}.
\]
If the map $V\longrightarrow dJ(\Omega;V)$ is linear and continuous with respect to 
the topology of $\mathcal D(D,\R^2):=C^\infty_c(D,\R^2)$, then $J$ is said to be shape differentiable at $\Omega$ and $dJ(\Omega; V)$ is called the shape derivative of $J$.\\
\end{definition} 

%%%%%%%%%%%%%%%%%%%%%%%%%%%%%%%%%%%%%%%%%%%%%%%%%
\subsection{An abstract differentiability result}
%%%%%%%%%%%%%%%%%%%%%%%%%%%%%%%%%%%%%%%%%%%%%%%%%

Let $E$ and $ F$ be Banach spaces. Let $G$ be a function 
\ben\label{eq:lagrangian}
G:[0,\tau]\times E \times F \rightarrow \R,\quad (t,\varphi,\psi)\mapsto G(t,\varphi,\psi)
\een
and define
\ben
E(t) :=\{u\in E|\; d_\psi G(t,u,0;\hat{\psi})=0\;\text{ for all } \hat{\psi}\in  F\}.
\een
Let us introduce the following hypotheses.
%%%%%%%%%%%%%%%%%%
\begin{assumption}[H0] 
For every $(t,\psi)\in [0,\tau]\times F$, we assume that
\begin{enumerate}
\item[(i)] the set $E(t)$ is single-valued and we write $E(t)= \{u^t\}$,
\item[(ii)] the function $[0,1]\ni s\mapsto G(t,su^t+s(u^t-u^0),\psi)$ is absolutely continuous,
\item[(iii)] the function $[0,1]\ni s\mapsto d_\varphi G(t,su^t+(1-s)u^0,\psi;\varphi)$ belongs to $L^1(0,1)$ for all $\varphi\in E$,
\item[(iv)] the function $\tilde\psi\mapsto G(t,u,\tilde\psi)$ is affine-linear.
\end{enumerate}
\end{assumption}
For $t\in [0,\tau]$ and $u^t \in E(t)$, let us introduce the set
\begin{equation}\label{averated_}
Y(t,u^t,u^0):=\left\{q\in  F|\;  \forall \hat{\varphi}\in E:\; \int_0^1 d_\varphi G(t,su^t+(1-s)u^0,q;\hat{\varphi})\, ds =0 \right\},
\end{equation}
which is called solution set of the \textit{averaged adjoint equation} with respect to $t$, $u^t$ and $u^0$.
Note that $Y(0,u^0,u^0)$ coincides with the solution set of the usual adjoint state equation:
\begin{equation}\label{averated_-1}
Y(0,u^0,u^0)= \left\{q\in  F|\; d_\varphi G(0,u^0,q;\hat{\varphi}) =0\quad \text{ for all }\hat{\varphi}\in E \right\}. 
\end{equation}
The following result, proved in \cite{sturm} allows us to compute the Eulerian semi-derivative of Definition \ref{def1} without computing the material derivative $\dot u$.
The key is the introduction of the set \eqref{averated_}.
% %%%%%%%%%%%%%%%%%%%%%%%%
\begin{theorem}\label{thm:sturm}
Let Assumption (H0) hold and the following conditions be satisfied.
\begin{itemize}
 \item[(H1)] For all $t\in [0,\tau]$ and all $\psi\in F$ the derivative $\partial_t G(t,u^0,\psi)$ exists.
 \item[(H2)] For all $t\in[0,\tau]$, the set $Y(t,u^t,u^0)$ is single-valued and we write $Y(t,u^t,u^0) = \{p^t\}$. 
 \item[(H3)] For any sequence of non-negative real numbers $(t_n)_{n\in \N}$ converging to zero, there exists a subsequence $(t_{n_k})_{k\in \N}$
such that
$$ \lim_{\stackrel{k\rightarrow \infty}{	s\searrow 0}}\partial_t G(s,u^0,p^{t_{n_k}})=\partial_tG(0,u^0,p^0).$$
\end{itemize}
Then for $\psi \in F$
we obtain
\begin{equation}\label{eq:dt_G_single}
\dt(G(t,u^t,\psi))|_{t=0}=\partial_t G(0,u^0,p^0).
 \end{equation}
\end{theorem}

%%%%%%%%%%%%%%%%%%%%%%%%%%%%%%%%%%%%%%%%%%%%%%%%%%%%%%%%
\subsection{Adjoint equation}
%%%%%%%%%%%%%%%%%%%%%%%%%%%%%%%%%%%%%%%%%%%%%%%%%%%%%%%%
Introduce the Lagrangian associated to the minimization problem \eqref{opt} for all $\varphi , \psi \in H^1_{0}(D)$:  
\begin{equation}
\label{lagrangean_non}
G(\Omega,\varphi ,\psi):=\int_{\Gamma_{0}} |B_r(\varphi)-B_d|^2\,ds+\int_{D} \beta(x,|\nabla u|^2)\nabla \varphi\cdot \nabla \psi\, dx -  \langle f,\psi\rangle
\end{equation}
The adjoint state equation is obtained by differentiating $G$ with respect to $\varphi$ at $\varphi = u$ and $\psi=p$,
\[
d_{\varphi}G(\Omega,u,p;\varphi)=0 \quad  \text{ for all } \varphi \in H^1_{0}(D),
\]
or, equivalently,
\begin{equation}
\label{adj}
\begin{aligned}
&2\int_{D}\partial_\zeta\beta(x,|\nabla u|^2) (\nabla u\cdot \nabla \varphi)( \nabla u\cdot \nabla p )\,dx +\int_{D} \beta(x,|\nabla u|^2)\nabla p\cdot \nabla \varphi \,dx \\
&=-2\int_{\Gamma_{0}}(B_r(u)-B_d)B_r(\varphi)\,ds \quad \text{ for all } \varphi \in H^1_{0}(D).
\end{aligned}
\end{equation}
Introduce the mean curvature $\kappa$ of $\Gamma_{0}$ and the Laplace-Beltrami operator $\Delta_{\tau}$ on $\Gamma_0$: 
\[
\Delta_{\tau} u  :=\Delta u- \kappa\frac{\partial u}{\partial n}-\frac{\partial^2 u}{\partial n^2}
\]
Using $B_r(u) = \nabla_{\tau} u$ as well as the equalities
\begin{align*}
(\nabla u\cdot \nabla \varphi)( \nabla u\cdot \nabla p )  & = ((\nabla u\otimes\nabla u) \nabla p)\cdot \nabla \varphi, \\
\int_{\Gamma_{0}}(\nabla_{\tau}  u-B_d)\cdot  \nabla_{\tau}\varphi\,   ds  & = -\int_{\Gamma_{0}}\Delta_{\tau}  u \varphi-\varphi \kappa B_{d}\cdot n\,ds 
\end{align*}
 and Green's formula, we deduce the corresponding strong form of \eqref{adj} 
\begin{align}\label{adjoint}
\begin{split}
-\divv(\ap_1(\nabla u)\nabla p)  &= 0 \quad   \text{ in }  \Omega,\\
-\divv(\ap_2(\nabla u)\nabla p)  &= 0 \quad    \text{ in }  D\setminus{\overline{\Omega}},\\
p&=0  \quad \text{ on } \partial D,
\end{split}
\end{align}
 with the transmission conditions
\begin{align}\label{trans_cond_adjoint}
\begin{split}
\left[p\right]_{\Gamma}  &= 0  \quad  \text{ on }   \Gamma, \\
\left[\ap(\nabla u)\nabla p\cdot n\right]_{\Gamma}&=0\quad \text{ on }  \Gamma,\\
\left[\ap(\nabla u)\nabla p\cdot n\right]_{\Gamma_{0}}&=
2\left( \Delta_{\tau} u  -\kappa B_{d}\cdot n\right) \text{ on }  \Gamma_{0}, 
\end{split}
\end{align}
where
\begin{align*}
\ap(\nabla u) & := \ap_1(\nabla u)\chi_{\Omega}  + \ap_2(\nabla u)\chi_{D\setminus\Omega},\\
\ap_i(\nabla u) & := \beta_{i}(|\nabla u|^2) I_2 +2\partial_{\zeta}\beta_{i}(|\nabla u|^2) \nabla u\otimes\nabla u \in\R^{2,2}, \quad i=1,2.
\end{align*}
Note that, with this notation, the variational form of the equation can be written as
\begin{equation}
\label{adj2}
\begin{aligned}
\int_{D} \ap(\nabla u)\nabla p\cdot\nabla \varphi \,dx =  -2\int_{\Gamma_{0}}(B_r(u)-B_d)B_r(\varphi)\,ds \quad \text{ for all } \varphi \in H^1_{0}(D).                       
\end{aligned}
\end{equation}

Now let us investigate the existence of a solution for the adjoint equation
	
\begin{theorem} \label{thm:exAdjoint}
Let Assumption \ref{assump:beta}.\ref{assump:beta_4} hold.
For given $u \in H^1_0(D)$ the equation
\begin{equation}
\label{adj_thrm}
	\begin{aligned}
		&\int_{D} \ap(\nabla u)\nabla p\cdot \nabla \varphi \,dx =-2\int_{\Gamma_{0}}(B_r(u)-B_d)B_r(\varphi)\,ds   \quad \text{ for all } \varphi \in H^1_{0}(D).                       
	\end{aligned}
\end{equation}
has a unique solution $p \in H^1_0(D)$.
\end{theorem}
\begin{proof}
For fixed $u\in H^1_0(D)$, define the bilinear form 
\begin{align*}
	a'(u;\cdot, \cdot): H^1_0(D) \times H^1_0(D) &   \rightarrow \mathbb R \\
	 (v,w) &     \mapsto \int_{D}  \ap(\nabla u) \nabla v\cdot \nabla w\,dx .
\end{align*}
We check the conditions of Lax-Milgram's theorem. The ellipticity of the bilinear form $a'(u; \cdot, \cdot)$ can be seen as follows:
\begin{align*}
	a'(u; v,v) &= \int_{D}  \ap(\nabla u) \nabla v\cdot \nabla v \,dx \geq \lambda \int_D |\nabla v |^2 \mbox d x \geq \lambda \, C \lVert v \rVert_{H^1(D)}^2
\end{align*}
where we have used the first estimate in Assumption \ref{assump:beta}.\ref{assump:beta_4}. and Poincar\'e's inequality since $v \in H^1_0(D)$.
The boundedness of the bilinear form $a'(u; \cdot, \cdot)$ can be seen by H\"older's inequality and again Assumption \ref{assump:beta}.\ref{assump:beta_4}.
The right-hand side is obviously a linear and continuous functional on $H^1_0(D)$,
\begin{align*}
	\langle F_u, \varphi \rangle = -2\int_{\Gamma_{0}}(B_r(u)-B_d)B_r(\varphi)\,ds,
\end{align*}
thus the theorem of Lax-Milgram yields the existence of a unique solution $p$ to the variational problem
\begin{align*}
	a'(u; \varphi, p) = \langle F_u, \varphi \rangle	\quad \text{ for all } \varphi \in H^1_0(D).
\end{align*}
\end{proof}

%%%%%%%%%%%%%%%%%%%%%%%%%%%%%%%%%%%%%%%%%%%%%%%%%%%%%%
\section{Shape derivative of the cost function} \label{sec:shapeDerivativeCostFunc}
%%%%%%%%%%%%%%%%%%%%%%%%%%%%%%%%%%%%%%%%%%%%%%%%%%%%%%
In this section we prove that the cost function $J$ given by \eqref{opt} is shape differentiable in the sense of Definition \ref{def1}. Moreover, we derive a domain expression of the shape derivative. To be more precise, Theorem~\ref{thm:sturm} is applied to show Theorem \ref{thm1}. Anticipating on the application of Section \ref{sec:numerics}, we assume in what follows that $f$ has the form
 \[
 f = f_0 + \operatorname{div}M\mbox{ with }M = M_1\chi_{\Omega_{\text{ma}}}(x)+ M_2\chi_{D\setminus \Omega_{\text{ma}}}(x)
 \]
where $f_0\in H^1(D)$.

In this section we assume $\Omega\subset \Omega^{\text{ref}}$, $V\in \mathcal{D}^1(\R^2,\R^2)$ 
and $\operatorname{supp}(V)\cap \Gamma_0 = \emptyset$.
Denote $\Omega^{\text{ref}}_k$, $k=1,..,8$, the connected components of $\Omega^{\text{ref}}$ (see Fig. \ref{fig1}). 
Introduce $\Gamma^{\text{ref}}$ the boundary of $\Omega^{\text{ref}}$. 
The four sides of $\Gamma^{\text{ref}}$ are denoted 
$\Gamma_k^{\text{ref},N},\Gamma_k^{\text{ref},W},\Gamma_k^{\text{ref},E},\Gamma_k^{\text{ref},S}$
where the exponents mean ``north'', ``south'', ``east'', ``west'', respectively. We assume $V\cdot n = 0$ on $\Gamma^{\text{ref},S}_k $ and $V\cdot n \leq 0$ on $\Gamma^{\text{ref},E}_k \cup \Gamma^{\text{ref},W}_k \cup \Gamma^{\text{ref},N}_k$. These conditions guarantee that $\Omega_t:=T_t(\Omega)\subset\Omega^{\text{ref}}$. In addition, we assume that the vector field $V$ is such that the transformation $T_t$ satisfies $T_t(\Oma) = \Oma$ for $t$ small enough. 

%%%%%%%%%%%%%%%%%%%%%%%%%%%%%%%%%%
\begin{theorem}\label{thm1} 
Let $\beta_{1}$ and $\beta_{2}$ satisfy Assumption \ref{assump:beta}. Then the functional $J$ is shape differentiable and its shape derivative in the direction $V$ is given by 

\begin{equation}
\label{volform}
\begin{aligned}
dJ(\Omega;V) =&  - \int_{D} (f_0 \operatorname{div}(V) +   \nabla f_0\cdot V) p\,dx \\
& +\int_{\Oma}   \PP'(0) \nabla p \cdot M_1  \, dx + \int_{D\setminus\overline{\Oma}}   \PP'(0) \nabla p \cdot M_2   \, dx\\
& +\int_{D} \beta_\Omega(x,|\nabla u|^2) \QQ'(0)\nabla u\cdot \nabla p \,dx\\
&-\int_{D} 2 \partial_\zeta \beta_\Omega(x,|\nabla u|^2) (D V^T\nabla u \cdot \nabla u) (\nabla u\cdot \nabla p)\, dx\\
\end{aligned}
\end{equation}
where $\PP'(0)= (\divv V) I_2-DV^T$, $\QQ'(0)= (\divv V) I_2-DV^T-DV$, $I_2 \in \mathbb R^{2,2}$ is the identity matrix, and
 $u,p\in H^1_0(D)$ are respectively the solutions of the problems \eqref{state}-\eqref{trans_cond_state} and \eqref{adjoint}-\eqref{trans_cond_adjoint}. 
 \end{theorem}
%%%%%%%%%%%%%%%%%%%%%%%%%%%%%%%%%%%%%%%%%%
\begin{remark}
Note that the last integral in \eqref{volform} is well-defined thanks to Assumption~\ref{assump:beta}.\ref{assump:beta_4}.
To see this, note that $V\in C^1(\overline D, \R^2)$, and that, for all $\zeta\in \R^2$, we have
$ \beta'(|\zeta|^2)|\zeta|^2 \le \Lambda$.
Hence 
\ben
\begin{split}
\left\lvert \int_{D} 2 \partial_\zeta \beta_\Omega(x,|\nabla u|^2) (D V^T\nabla u \cdot \nabla u) (\nabla u\cdot \nabla p)\, dx \right\rvert &\le C \int_{D} \partial_\zeta \beta_\Omega(x,|\nabla u|^2)|\nabla u|^2 |\nabla u\cdot \nabla p|\, dx \\
        & \le C \Lambda \int_{D} |\nabla u\cdot \nabla p|\, dx <\infty .
\end{split}
\een
The other terms in \eqref{volform} are obviously well-defined.
\end{remark}
%%%%%%%%%%%%%%%%%%%%%%%%%%%%%%%%%%%%%%%%%%%%%%%%
%%%%%%%%%%%%%%%%%%%%%%%%%%%%%%%%%%%%%%%%%%%%%%%%
%%%%%%%%%%%%%%%%%%%%%%%%%%%%%%%%%%%%%%%%%%%%%%%%
\noindent {\bf Proof of Theorem \ref{thm1}.}
%%%%%%%%%%%%%%%%%%%%%%%%%%%%%%%%%%%%%%%%%%%%%%%%
%%%%%%%%%%%%%%%%%%%%%%%%%%%%%%%%%%%%%%%%%%%%%%%%
Let us consider the transformation $T_{t}$ defined by \eqref{transf} with $V\in \mathcal{D}^1(D,\R^2)$. 
In this case, $T_{t}(D)=D$, but, in general, $T_{t}(\Omega):=\Omega_{t}\neq \Omega$. We define the Lagrangian $G(\Omega_{t},\varphi,\psi)$ at the transformed domain $\Omega_t$ for all $\varphi,\psi$ in $H^1_0(D)$: 
\[
\begin{aligned}
G(\Omega_{t},\varphi,\psi) &=\int_{\Gamma_{0}} |B_{r}(\varphi )-B_d|^2\;ds+\int_{D} \beta_{\Omega_t}(x, |\nabla \varphi  |^2)\nabla \varphi \cdot \nabla \psi \;dx  - \langle f,\psi \rangle  .
\end{aligned}
\]
Since 
\[
 f = f_0 + \operatorname{div}M\mbox{ with }M = M_1\chi_{\Oma}(x)+ M_2\chi_{D\setminus\overline\Oma}(x)
\]
where $M_1$ and $M_2$ are constant vectors, we transform the last term in $G$ to
\begin{align*}
\langle f,\psi \rangle  & = \int_D f_0 \psi + \langle \divv M,\psi \rangle  = \int_D f_0 \psi  - \int_D M\cdot\nabla\psi,
\end{align*}
which yields
\begin{align*}
G(\Omega_{t},\varphi,\psi) &=\int_{\Gamma_{0}} |B_{r}(\varphi )-B_d|^2\;ds+\int_{D} \beta_{\Omega_t}(x, |\nabla \varphi  |^2)\nabla \varphi \cdot \nabla \psi \;dx  \\
& - \int_D f_0 \psi  + \int_{\Oma} M_1\cdot\nabla\psi  + \int_{D\setminus\overline\Oma} M_2\cdot\nabla\psi.
\end{align*}
In order to differentiate $G(\Omega_{t},\varphi,\psi)$ with respect to $t$, the integrals in $G(\Omega_t,\varphi,\psi)$ need to be transported back on the reference domain $\Omega$ using the transformation $T_t$. However, composing by $T_t$ inside the integrals creates terms $\varphi\circ T_t$ and $\psi\circ T_t$ which might be non-differentiable. To avoid this problem, we need to parameterize the space $H^1_{0}(D)$ by composing the elements of $H^1_{0}(D)$ with $T_t^{-1}$. Following this argument, we introduce  
\ben
\label{lagrangeant} 
\mathfrak{G}(t,\varphi,\psi ):=G(\Omega_{t}, \varphi\circ T^{-1}_{t},\psi\circ T^{-1}_{t} ).
\een
In \eqref{lagrangeant}, we proceed to the change of variable $x = T_t(\bar x)$. 
This yields 
\ben\label{lagrangean_non_repa}
\begin{split}
\mathfrak{G}(t,\varphi,\psi)&=
    \int_{\Gamma_{0}} \xi_\tau(t)|B_{r}(\varphi)-B_d\circ T_{t}|^2\;ds\\
&  +\int_{D} \beta_{\Omega}(x,|\BB (t)\nabla \varphi|^2)\BB (t)\nabla \varphi\cdot \BB (t)\nabla \psi\, \xi(t) \;dx \\
& - \int_{D}f_{0}\circ T_t \psi \xi(t) \;dx  
+   \int_{\Oma} M_1 \cdot \BB(t) \nabla\psi  \xi(t)\, dx\\
& +   \int_{D\setminus\overline{\Oma}} M_2 \cdot \BB(t) \nabla\psi\, \xi(t)\, dx,
\end{split}
\een
where $\BB (t)= DT_{t}^{-T}$ and $\xi(t),\xi_\tau(t)$ are defined in \eqref{JacobiDet} and \eqref{tangJacobiDet}, respectively. Note that we have used the assumption $T_{t}(\Oma) = \Oma$ in the computation of $\mathfrak{G}(t,\varphi,\psi)$.

Note that $J(\Omega_t)=\mathfrak{G}(t, u^t, \psi)$ for all $\psi \in H^1_0(D)$, where $u^t\in H^1_0(D)$ solves 
\ben\label{eq:for_u_t_trans_problem}
\begin{split}
\int_D & \beta_{\Omega}(x,|\BB (t)\nabla u^t|^2)\BB (t)\nabla u^t \cdot  \BB (t)\nabla \psi \xi(t)  \;dx   \\
& =\int_D  f_{0}\circ T_{t} \psi \xi(t) \;dx 
- \int_{\Oma} M_1 \cdot \BB(t) \nabla\psi   \xi(t)\, dx- \int_{D\setminus\overline\Oma} M_2 \cdot \BB(t) \nabla\psi  \xi(t)\, dx.
\end{split}
\een
To prove Theorem \ref{thm1}, we need to check the conditions of Theorem~\ref{thm:sturm} for the function $\mathfrak{G}(t,\varphi,\psi)$ with $E=F=H^1_0(D)$.

{\bf Verification of (H0).} Condition (H0)(i) is satisfied since $E(t)=\{u^t\}$, where $u^t\in H^1_0(D)$ is the solution of the state equation \eqref{eq:for_u_t_trans_problem}.
Conditions (ii) and (iii) of Assumption (H0) are also satisfied due to the differentiability of the functions $\beta_1,\beta_2$ and
Assumption 1. Condition (H0)(iv) is satisfied by construction. 

{\bf Verification of (H1).} Condition (H1) is satisfied since $\BB (t)$, $\xi(t)$ and $\xi_\tau(t)$ are smooth. 

{\bf Verification of (H2).} $Y(t,u^t,u^0)=\{p^t\}$, where $p^t\in H^1_0(D)$ is the unique solution of 
\ben\label{eq:linearization_transm_in_u}
\begin{split}
\int_0^1\int_{D}& 2 \partial_{\zeta}\beta_{\Omega}(x,|\BB (t)\nabla u^s_t|^2) ((\BB (t)\nabla u^s_t \otimes \BB (t)\nabla u^s_t)\BB (t)\nabla p^t)\cdot \BB (t)\nabla\psi) \xi(t)\;dx\,ds\\
                      &    +\int_0^1\int_{D} \beta_{\Omega}(x,|\BB (t)\nabla u^s_t|^2)\BB (t)\nabla \psi\cdot \BB (t)\nabla p^t \xi(t)\;dx\,ds\\
                             &= - \int_0^1\int_{\Gamma_0} 2(B_{r}(u^s_t)-B_d)B_{r}(\psi) \xi(t) \;dx\,ds \text{ for all } \psi \in H^1_0(D).
\end{split}
\een
which can also be rewritten in a more compact way as
\ben\label{eq:linearization_transm_in_u2}
\begin{split}
\int_0^1\int_{D}& \mathcal{A}(\BB (t) \nabla u^s_t ) \BB (t) \nabla p^t\cdot \BB (t)\nabla \psi \xi(t)\;dx\,ds  \\
& = - \int_0^1\int_{\Gamma_0} 2(B_{r}(u^s_t)-B_d)B_{r}(\psi) \xi(t) \;dx\,ds.
\end{split}
\een
To prove that the previous equation has indeed a unique solution, we first check that all integrals are finite in the previous 
equation. To verify this, we use H\"older's inequality to obtain 
\begin{equation*}
\begin{split}
\int_{D}&2 \partial_{\zeta}\beta_{\Omega}(x,|\BB (t)\nabla u^s_t|^2) ((\BB (t)\nabla u^s_t \otimes \BB (t)\nabla u^s_t)\BB (t)\nabla p^t)\cdot \BB (t)\nabla\psi) \xi(t)\;dx \\
  \le  c   &\left(\int_D 2\partial_{\zeta} \beta_{\Omega}(x,|\BB (t)\nabla u^s_t|^2) (\BB (t)\nabla u^s_t\cdot \BB (t)\nabla p^t)^2 \,dx \right)^{1/2}\cdot\\
  &\qquad \left(\int_D 2\partial_{\zeta}\beta_{\Omega}(x,|\BB (t)\nabla u^s_t|^2)(\BB (t)\nabla u^s_t\cdot \BB (t)\nabla\psi)^2 \;dx \right)^{1/2}\\ 
 \end{split}
\end{equation*}
and 
\begin{equation*}
\begin{split}
 \int_{D}& \beta_{\Omega}(x,|\BB (t)\nabla u^s_t|^2)\BB (t)\nabla \psi\cdot \BB (t)\nabla p^t \xi(t)\;dx\\
 &\le c\left(\int_{D} \beta_{\Omega}(x,|\BB (t)\nabla u^s_t|^2)|\BB (t)\nabla \psi|^2 \,dx \right)^{1/2} \left(\int_{D} \beta_{\Omega}(x,|\BB (t)\nabla u^s_t|^2) | \BB (t)\nabla p^t|^2 \;dx\right)^{1/2}.
\end{split}
\end{equation*}
Adding both equations and using part \ref{assump:beta_4} of Assumption \ref{assump:beta}, we get
\ben%\label{eq:linearization_transm_in_u}
\begin{split}
\int_{D}& \mathcal{A}(\BB (t) \nabla u^s_t ) \BB (t) \nabla p^t\cdot \BB (t)\nabla \psi \xi(t)\;dx\,ds \le c \|\psi\|_{H^1(D)}\|\bar p^t\|_{H^1(D)},
\end{split}
\een
where the constant $c>0$ is independent of $s$.

The existence of a solution $p^t$ follows from Theorem \ref{thm:exAdjoint}. Since $\BB (t) = DT_t^{-T}$, there are numbers $C>0$ and $\tau>0$ such that, for all $t\in [0,\tau]$ and $\rho\in \R^2$, we have 
$\BB (t)\rho\cdot \rho\ge C |\rho|^2$.
Note that $p^0=p\in Y(0,u^0,u^0)$ is the unique solution of the adjoint equation \eqref{adjoint}-\eqref{trans_cond_adjoint}. 

{\bf Verification of (H3).} To verify this assumption we show that there is a sequence $(p^{t_k})_{k\in \N}$, where $\{p^{t_k}\} = Y(t_k,u^{t_k},u^0)$, which converges weakly in $H^1_0(D)$ to the solution of the adjoint equation and that $(t,\psi)\mapsto \partial_t \mathfrak G(t,u^0,\psi)$ is weakly continuous.
In order to prove this, we need the following lemmas. 

\begin{lemma}
\label{shape-tools}
Let $m\in \{0,1\}$ and the velocity field $V\in \mathcal{D}^1(\R^2,\R^2)$ be given and $\varphi\in H^m(\R^2)$.
We denote by $T_{t}$, the transformation associated to $V$. Then we have 
\[
\lim_{t\to 0}\varphi\circ T_{t}=\varphi \text{ and }\lim_{t\to 0}\varphi\circ T^{-1}_{t}=\varphi \quad \text{ in } H^m(D).
\]
\end{lemma}
\begin{proof}
 See for instance \cite{DZ2}. 
 \end{proof}

Recall that according to Remark~\ref{rem:lipschitz} there are constants $C,\ve>0$ such that 
\begin{equation}\label{eq:lipschitz}
\| u(\chi_{1})-u(\chi_{2})\|_{H^1(D)}\leq \| \chi_{1}-\chi_{2} \|_{L^1(D)},\quad  \forall  \chi_{1}, \chi_{2}\in \Xi(D),
\end{equation}
where 
\[
\Xi(D):=\left\{\chi_{\Omega}:  \Omega \subset D\text{ is measurable and }  \chi_{\Omega}(1-\chi_{\Omega})=0 \text{  a.e. in } D\right\}.
\]
With this result it is easy to see that $t\mapsto u^t$ is in fact continuous.
\begin{lemma}\label{lem:mate_der_trans}
Let Assumptions~\ref{assump:beta}.\ref{assump:beta_1} to \ref{assump:beta}.\ref{assump:beta_3} hold. Then the mapping $t\mapsto u^t:=\Psi_t(u_t)$ is continuous from the right in $0$, i.e., 
$$ \lim_{t\searrow 0} \|u^t -  u\|_{H^1_0(D)} =0. $$
If in addition Assumption~\ref{assump:beta}.\ref{assump:beta_4} is satisfied, then there are constants $c,\tau>0$ such that
$$\|u^t-u\|_{H^1_0(D)} \le tc, \quad \text{ for all } t\in [0,\tau].  $$
\end{lemma}
%%%%%%%%%%%%%%%%%%%%%%%%%%%%%%%%
\begin{proof} Since $\|\,\cdot\, \|_{H^1_0(D)}$ and 
the $\Ltwo$-norm of the gradient are equivalent norms on $H^1_0(D)$ it suffices to show
$\lim_{t\searrow 0} \|\nabla u^t-\nabla u\|_{\Ltwo(D,\R^d)} = 0 . $ 
First of all introduce 
$$A_t(x) :=\xi(t)(D\Phi_t(x))^{-T}(D\Phi_t(x))^{-1}$$
which satisfies $|A_t(x)^{-1}| \le c$ for all $t\in [0,\tau]$ and hence
\ben
\begin{split}
\forall x\in \Omega,\, \forall t\in [0,\tau],\,\forall \zeta\in \R^2:\,  |\zeta|^2 & \le |A_t(x)^{-1}| \,|\xi(t) (D\Phi_t(x))^{-T}\zeta\cdot (D\Phi_t(x))^{-T}\zeta|^2  \\
                                                            &  =   c A_t(x)\zeta\cdot \zeta. 
\end{split}
\een
Therefore for all $f_{0}\in H^1(D)$ and all $t\in[0,\tau]$
\benn
\int_D|\nabla (f_{0}\circ T_t^{-1})|^2\,dx=\int_D A_t \nabla f_{0}\cdot \nabla f_{0}\,dx \ge \frac{1}{c}\int_D|\nabla f_{0}|^2 \,dx.
\eenn 
Further, we get from this estimate that for all $t\in[0,\tau]$ 
 \ben\label{eq:int_D_f_phi_t}
c\|f_{0}\|_{H^1(D)}\le \|f_{0}\circ T_t^{-1}\|_{H^1(D)}.
\een
Now setting $\chi_1:=\chi_\Omega$ and $\chi_2:=\chi_{\Omega_t}=\chi_\Omega\circ T_t^{-1}$ and 
denoting the corresponding solutions of \eqref{weak_state} by $u:=u(\chi_\Omega)$ and $u_t:=u(\chi_{\Omega_t})$, we infer from 
\eqref{eq:lipschitz} and \eqref{eq:int_D_f_phi_t}
$$c\|\nabla(u\circ T_t)-\nabla u^t\|_{\Ltwo(D,\R^d)}\le \|\nabla u_t-\nabla u\|_{\Ltwo(D,\R^d)}\le C\|\chi_\Omega-\chi_\Omega\circ  T_t^{-1}\|_{L^1{(D)}},$$
where $u^t:=u_t\circ  T_t$. 
Employing the previous estimates, we get for all $t\in [0,\tau]$
\benn
\begin{split}
\|\nabla u-\nabla u^t\|_{\Ltwo(D,\R^d)}&\le \|\nabla u-\nabla (u\circ T_t)\|_{\Ltwo(D,\R^d)}+\|\nabla (u\circ T_t)-\nabla u^t\|_{\Ltwo(D,\R^d)}\\
                  &\le \tilde{c}\left(\|\nabla u-\nabla (u\circ T_t)\|_{\Ltwo(D,\R^d)}+\|\chi_\Omega-\chi_\Omega\circ  T_t^{-1}\|_{L^1{(D)}}\right).
\end{split}
\eenn
Finally, taking into account Lemma~\ref{shape-tools}, we obtain the desired continuity. 
The Lipschitz continuity under the additional Hypothesis~\ref{assump:beta}.\ref{assump:beta_4} was shown in \cite{sturm}.
\end{proof}
%
%%%%%%%%%%%%%%%%%%%%%%%%%%%%%%%%%%%%%%%%%%%%%%%%%%%%%%%%
Using the previous lemma we are able to show the following.
\begin{lemma}
The solution $p^t$ of \eqref{eq:linearization_transm_in_u} converges weakly in $H^1_0(D)$ to the solution $p$ of the adjoint equation \eqref{adjoint}-\eqref{trans_cond_adjoint}.
\end{lemma}
%%%%%%%%%%%%%%%%%%%%%%%%%%%%%%%%%%%%%%%%%5
\begin{proof}
The existence of a solution of \eqref{eq:linearization_transm_in_u} follows from Theorem \ref{thm:exAdjoint}. Inserting $\psi=p^t$ as test function 
in \eqref{eq:linearization_transm_in_u}, we see that the estimate $\|u^t\|_{H^1(D)}\le C$ implies $\|p^t\|_{H^1(D)}\le C$ for $t$ sufficient small. From the
boundedness, we infer that $(p^t)_{t\ge 0}$ converges weakly to some $w\in H^1_0(D)$. In Lemma~\ref{lem:mate_der_trans} we proved $u^t\rightarrow u$ in $H^1_0(D)$
which we can use to pass to the limit in \eqref{eq:linearization_transm_in_u} and obtain 
\benn
p^{t_k} \rightharpoonup p \text{ in } H^1_0(D) \text{ for } t_k\rightarrow 0 \text{ as }k\rightarrow \infty,
\eenn
where $p\in H^1_0(D)$ solves the adjoint equation \eqref{adjoint}-\eqref{trans_cond_adjoint}. By uniqueness we conclude $w=p$.
\end{proof}
Now we proceed to the differentiation of \eqref{lagrangean_non_repa} at $t>0$. 
Introduce the notations $\PP(t) =\xi(t)\BB(t) $ and $\QQ(t) : = \xi(t)\BB (t)^T\BB (t) $, we obtain
\ben\label{eq:volume_form_trans}
\begin{split}
 \partial_t \mathfrak{G}(t,\varphi,\psi) &= \int_{\Gamma_{0}} \xi_\tau'(t)|B_{r}(\varphi)-B_{d}\circ T_{t}|^2\;ds \\
& \hspace{-1cm} -2\int_{\Gamma_{0}} \xi_\tau(t) (B_{r}(\varphi)-B_{d}\circ T_{t}) \nabla B_d\circ T_t \cdot V   \;ds\\
& \hspace{-1cm}  +2\int_{D}  \partial_{\zeta}\beta_{\Omega}(x,|\BB (t)\nabla \varphi|^2)  (\BB '(t)\nabla \varphi\cdot \BB (t)\nabla \varphi)\BB (t)\nabla \varphi\cdot \BB (t)\nabla \psi \xi(t)\;dx\\
& \hspace{-1cm}  +\int_{D} \beta_{\Omega}(x,|\BB (t)\nabla \varphi|^2)\QQ'(t)\nabla \varphi\cdot \nabla \psi \;dx\\
& \hspace{-1cm}  -\int_{D} f_{0}\circ T_{t}\, \psi \xi'(t)  + \nabla f_{0}\circ T_{t}\cdot V\,  \psi \xi(t)\;dx\\
& \hspace{-1cm}  +\int_{\Oma}   \PP'(t) \nabla\psi \cdot M_1  \, dx + \int_{D\setminus\overline{\Oma}}   \PP'(t) \nabla\psi \cdot M_2   \, dx\\
\end{split}
\een
and this shows that for fixed $\varphi\in H^1_0(D)$ the mapping $(t,\psi)\mapsto \partial_t \mathfrak G(t,\varphi,\psi)$ is weakly continuous. 
This finishes proving that (H3) is satisfied.

Consequently, we may apply Theorem~\ref{thm:sturm} and obtain
$dJ(\Omega;V)=\partial_t \mathfrak{G}(0,u,p)$,
where $u\in H^1_0(D)$ solves the state equation \eqref{state}-\eqref{trans_cond_state} and $p\in H^1_0(D)$ is a solution of the adjoint equation \eqref{adjoint}-\eqref{trans_cond_adjoint}. 
In order to compute $\partial_t \mathfrak{G}(0,u,p)$, note that the integrals on $\Gamma_0$ vanish due to $V=0$ on $\Gamma_0$, $\BB'(0)= -DV^T$, $\PP'(0)= (\divv V) I_2-DV^T$, $\QQ'(0)= (\divv V) I_2-DV^T-DV$.

Finally we have proved formula \eqref{volform}. This concludes the proof of Theorem \ref{thm1}. \hfill $\blacksquare$

%%%%%%%%%%%%%%%%%%%%%%%%%%%%%%%%%%%%%%%%%%%%%%%%%%%%%%%%%%%%%%%
\section{Optimization of the rotor core} \label{sec:numerics}
%%%%%%%%%%%%%%%%%%%%%%%%%%%%%%%%%%%%%%%%%%%%%%%%%%%%%%%%%
In this section, we use the shape derivative derived in Theorem \ref{thm1} to determine the optimal design for the electric motor described in Section \ref{sec:problemFormulation}. Recall that the problem consists in finding the shape $\Omega\subset\Omega_f$ of the ferromagnetic subdomain of the electric motor depicted in Fig. \ref{fig1} which minimizes the cost functional 
\begin{align*}
	J(\Omega) = \int_{\Gamma_0} | B_r(u_{\Omega}) - B_d|^2 \mbox d s
\end{align*}
among all admissible shapes $\Omega \in \mathcal O$ where $\Gamma_0$ is a circle in the air gap, $B_r(u_{\Omega})$ denotes the radial component of the magnetic flux density $B = B(u_{\Omega})$ on $\Gamma_0$ and $B_d$ is a given sine curve, $B_d(\theta) = \frac{1}{2}\, \mbox{sin}(4 \, \theta)$ where $\theta$ denotes the angle in polar coordinates with origin at the center of the motor; see Fig. \ref{fig1}. Minimizing this functional leads to a reduction of the total harmonic distortion (THD; see \cite{binder2012elektrische,ChoiIzuiNishiwakiKawamotoNomura2012}) of the flux density which causes the rotor to rotate more smoothly. Here, $u_{\Omega}$ is the solution of the two-dimensional magnetostatic boundary value problem the weak form of which reads as follows:
\begin{align}
	\begin{aligned} \label{eq:state_weak}
		&\mbox{Find } u \in H^1_0(D) \mbox{ such that } \int_D \nu(|\nabla u|) \nabla u \cdot \nabla v \, \mbox d x = \langle f, v \rangle  \mbox{ for all } v \in H^1_0(D).
	\end{aligned}
\end{align}
Here, the right-hand side $f$ corresponds to the weak form of 
 \[
 f = f_0 + \operatorname{div}M\mbox{ with }M = M_1\chi_{\Omega_{\text{ma}}}(x)+ M_2\chi_{D\setminus \Omega_{\text{ma}}}(x)
 \]
as in Section \ref{sec:shapeDerivativeCostFunc}, where $f_0 = J_i$, $M_1 = -\frac{\nu}{\nu_0} \binom{-M_y}{M_x}$ and $M_2 = 0$. The vector $\binom{M_x}{M_y}$ denotes the permanent magnetization of the magnets. It is a constant vector in each of the magnet subdomains pointing in the directions indicated in Fig. \ref{fig1} and vanishes outside the magnet areas. Denoting $M^{\perp} = \frac{\nu}{\nu_0} \binom{-M_y}{M_x}$, the right hand side $f \in H^{-1}(D)$ reads
\begin{align}
	\langle f, v \rangle = \int_D J_i \, v \, + \, M^{\perp} \cdot \nabla v \, \mbox d x.
\end{align}
The function $J_i$ represents the impressed currents in the coil areas (light blue areas in Fig. \ref{fig1}) and is assumed to vanish in this special application, i.e., $J_i = 0$. 

We consider admissible shapes $\Omega\in\mathcal{O}$ as in \eqref{admissible} and in Section \ref{sec:shapeDerivativeCostFunc}. 
Furthermore,
\begin{align*}
	\nu(|\nabla u|) = \chi_{\Omega_f} \, \hat{\nu}(|\nabla u|) + \chi_{D \setminus \overline{\Omega_f}} \, \nu_0 
\end{align*}
 denotes the magnetic reluctivity composed of a nonlinear function $\hat{\nu}$ depending on the magnitude of the magnetic flux density $|B| = |\nabla u|$ inside the ferromagnetic material and a constant $\nu_0 = 10^7 / (4 \pi)$, which is expressed in the unit $m \, kg \, A^{-2} \, s^{-2}$, otherwise. The constant $\nu_0$ is the magnetic reluctivity of vacuum which is practically the same as that of air. The nonlinear function $\hat{\nu}$ is in practice obtained from measurements and is not available in a closed form. However, the physical properties of magnetic fields yield the following characteristics of $\hat{\nu}$:
 \renewcommand{\theenumi}{\roman{enumi}}%
\begin{enumerate}
	 \item[(i)] $\hat{\nu}$ is continuously differentiable on $(0, \infty)$,
	 \item[(ii)] $\exists \, m >0: \; \hat{\nu}(s) \geq m \quad \text{ for all } s \in \mathbb{R}_0^+$,
	 \item[(iii)] $\hat{\nu}(s) \leq \nu_0 \quad \text{ for all } s \in \mathbb{R}_0^+$,
	 \item[(iv)] $(\hat{\nu}(s)s)' = \hat{\nu}(s) + \hat{\nu}'(s)s \geq m>0$,
	 \item[(v)] $s \mapsto \hat{\nu}(s) s$ is strongly monotone with monotonicity constant $m$, i.e.,
	 \begin{align*}
		(\hat{\nu}(s)\, s - \hat{\nu}(t)\,t)\, (s-t) &\geq m (s-t)^2 \quad \text{ for all } s, t \geq 0,
	 \end{align*}
	 \item[(vi)] $s \mapsto \hat \nu(s) s$ is Lipschitz continuous with Lipschitz constant $\nu_0$, i.e.,
	 \begin{align*}
		|\hat{\nu}(s)\,s - \hat{\nu}(t)\,t)| &\leq \nu_0 |s-t| \text{ for all } \forall s, t \geq 0.
	 \end{align*}
\end{enumerate}
\renewcommand{\theenumi}{\arabic{enumi}}%
For more details on properties and practical realization of the function $\hat{\nu}$ from given measurement data, we refer the reader to \cite{Heise1994, JuettlerPechstein2006, Pechstein2004}.
\begin{figure}
	\begin{minipage}{7cm}
		\begin{tabular}{c}    
			\includegraphics[scale=0.65]{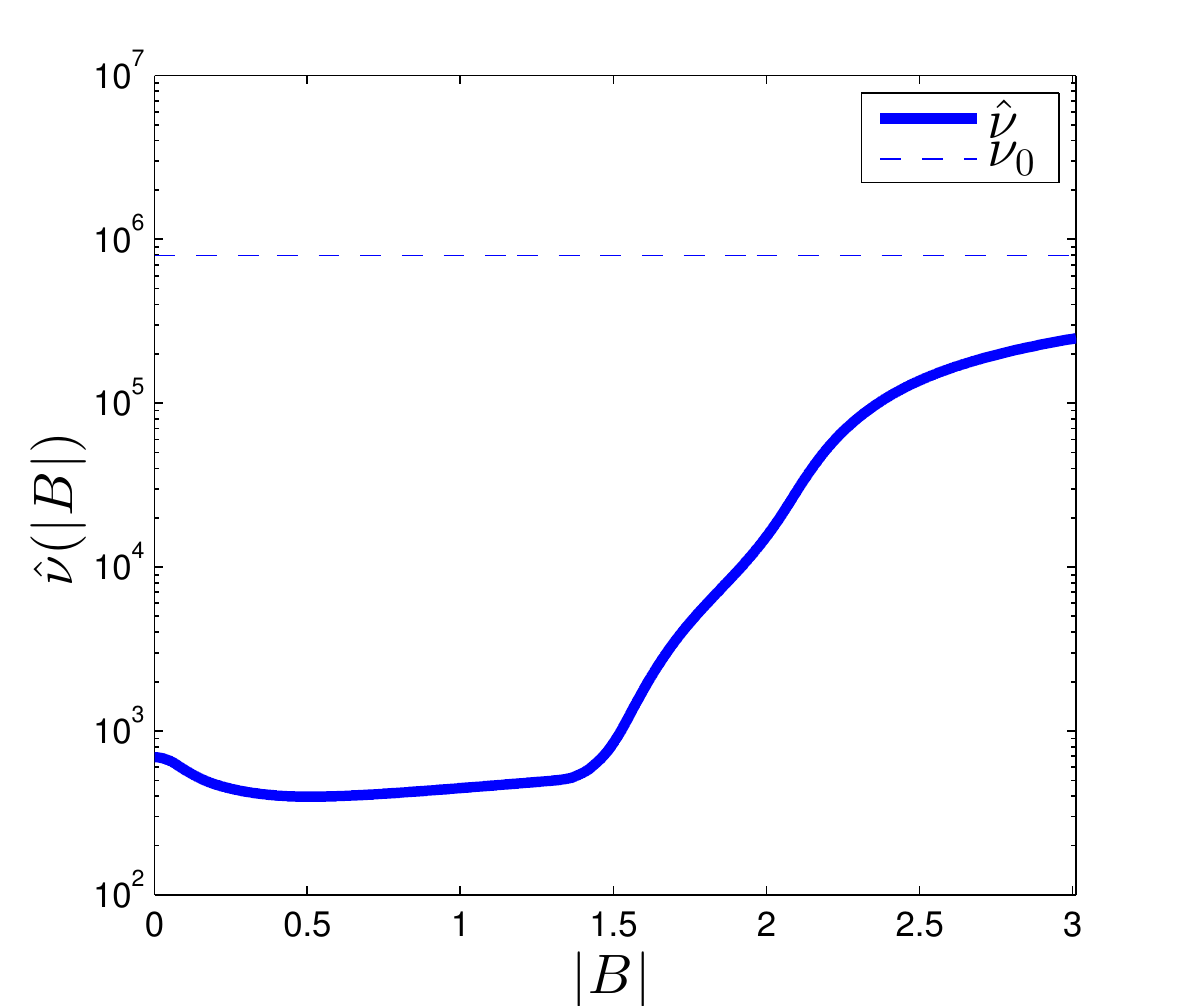} \\
			(a)
		\end{tabular}
	\end{minipage}
	\hspace{1.2cm}
	\begin{minipage}{7cm}
		\begin{tabular}{c}    
			\hspace{-10mm}\includegraphics[scale=0.4, trim=200 0 200 0, clip=true]{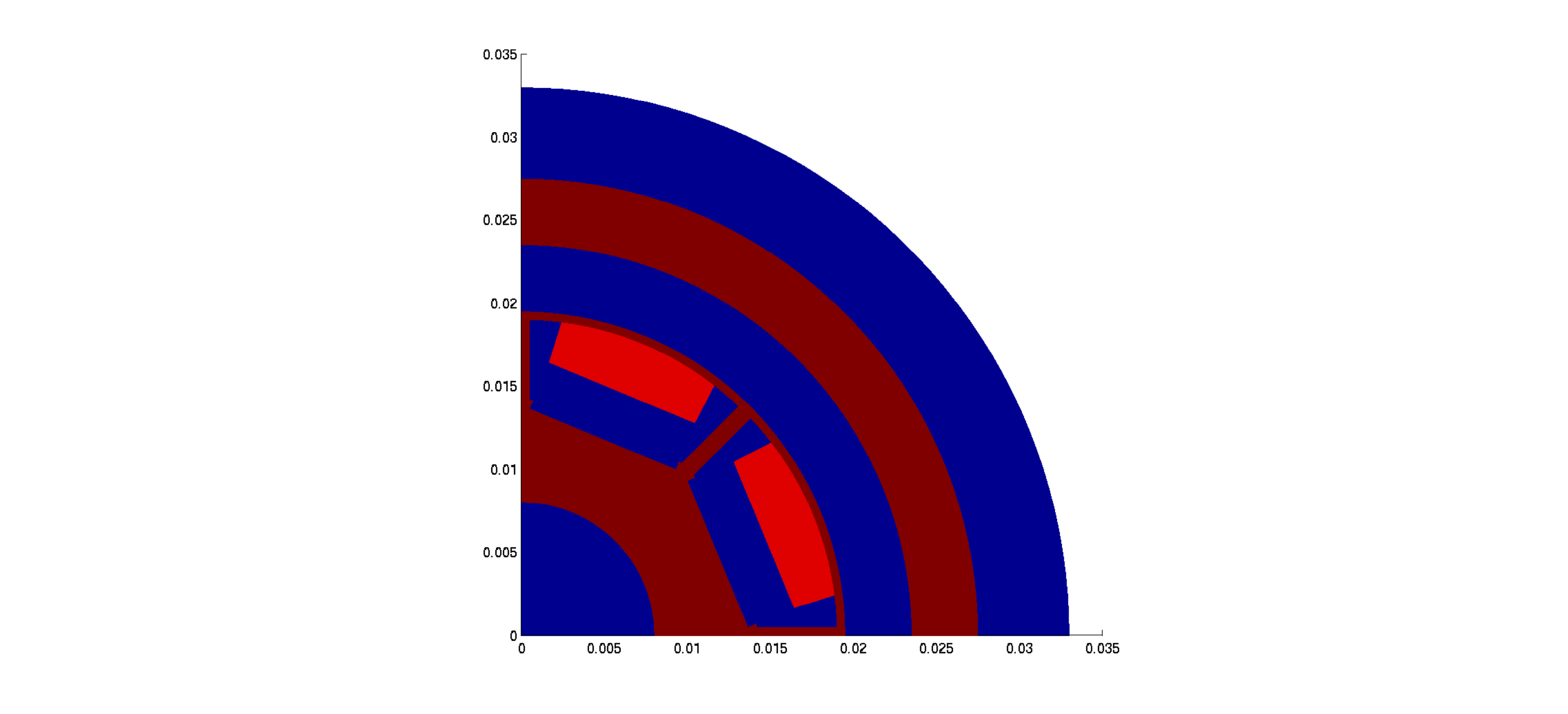}\\
			(b)
		\end{tabular}
	\end{minipage}
	\caption{(a) Magnetic reluctivity $\hat{\nu}$ as a function of the magnitude $|B|~=~|\nabla u|$ of the magnetic flux density. (b) Ferromagnetic material $\Omega_f$ with design subdomains $\Omega \subset \Omega_f$ (highlighted). }
	\label{fig:nu_designAreas}
\end{figure}

In order to be able to apply Theorem \ref{thm1} to the problem above, we have to check whether Assumption \ref{assump:beta} is satisfied for 
\begin{align*}
	\beta_1(\zeta) &:= \hat{\nu}(\sqrt{\zeta}) \mbox{ and } \beta_2(\zeta) := \nu_0.
\end{align*}
Clearly, all four conditions of Assumption \ref{assump:beta} are fulfilled for $\beta_2(\zeta) = \nu_0 = const$. Now let us investigate more closely $\beta_1$. Notice the relations $\beta_1( | \rho |^2) = \hat{\nu}(| \rho |)$ and $\beta_1'(| \rho |^2) = \hat{\nu}'(| \rho |) / (2| \rho |)$.
\begin{enumerate}
	\item As mentioned above, the function $\hat{\nu}$ is bounded from above by the magnetic reluctivity of vacuum $\nu_0$ and from below by a positive constant $m$, compare Fig. \ref{fig:nu_designAreas}(a).
	\item Assumption \ref{assump:beta}.\ref{assump:beta_2} holds for the function $\hat{\nu}$ by virtue of properties (v) and (vi).
% 	is verified for the discrete function $\hat{\nu}(|B|)$.
	\item The numerical realization of the function $\hat{\nu}$ consists in an interproximation of given measurement data. The interproximation was done using splines of class $C^1$.
	\item It is easy to see that this condition for $\beta_1$ is equivalent to
		\begin{align*}
			\exists \lambda, \, \Lambda >0: \lambda | \eta |^2 \leq \eta^T \left( \beta_1(|\rho|^2) I_2 + 2\beta_1'(| \rho |^2) \rho \otimes \rho \right) \eta \leq \Lambda | \eta |^2,
		\end{align*}
		or in terms of $\hat{\nu}$,
		\begin{align*}
			\exists \lambda, \, \Lambda >0: \lambda | \eta |^2 \leq \eta^T \left( \hat{\nu}(|\rho|) I_2 + \frac{\hat{\nu}'(| \rho |)}{| \rho |} \rho \otimes \rho \right) \eta \leq \Lambda | \eta |^2,
		\end{align*}
		where $I_2$ denotes the identity matrix of dimension 2. The eigenvalues and corresponding eigenvectors of the $2 \times 2$ matrix $\hat{\nu}(|\rho|) I_2 + \frac{\hat{\nu}'(| \rho |)}{| \rho |} \rho \otimes \rho$ are given by
		\begin{align*}
			\lambda_1 &= \hat{\nu}(|\rho |) & v_1 &= \binom{-\rho_2}{ \rho_1}, \\
			\lambda_2 &= \hat{\nu}(|\rho |) + \hat{\nu}'(| \rho | )| \rho | &v_2 &= \binom{\rho_1}{\rho_2}.
		\end{align*}
		From the physical properties (ii) and (iv) of $\hat{\nu}$ it follows that both $\lambda_1$ and $\lambda_2$ are positive. Therefore, the assumption holds with $\lambda = \mbox{min}\lbrace \lambda_1, \lambda_2 \rbrace$ and $\Lambda = \mbox{max}\lbrace \lambda_1, \lambda_2 \rbrace$.
\end{enumerate}
Properties (v) and (vi) together with Assumption \ref{assump:beta}.\ref{assump:beta_1} yield existence and uniqueness of a solution $u\in H^1_0(D)$ to the state equation. Assumption \ref{assump:beta}.\ref{assump:beta_4} ensures the existence of an adjoint state $p\in H^1_0(D)$.

Thus we can apply Theorem \ref{thm1} and the shape derivative is given by \eqref{volform}.

\subsection{Numerical method}
% % For the optimization, we apply the so-called \textit{velocity method} described in Section \ref{sec:shapDerivative}. 
In each iteration of the optimization process we use the shape derivative $dJ(\Omega; V)$ derived in \eqref{volform} to compute a vector field $V$ that ensures a decrease of the objective function $dJ(\Omega; V) \leq 0$ by displacing the interface between the iron subdomain $\Omega$ and the air subdomain $\Omr\setminus \Omega$ along that vector field.

\subsubsection{Setup of interface}
Due to practical restrictions we choose not to move the interface by moving the single points of the mesh, as it is common practice in shape optimization. Instead we model the interface by setting up a polygon with 151 points around each of the design subdomains $\Omega_k^{ref}$ (see Fig. \ref{fig:initInterface}) and move the points of these polygons along the calculated velocity field $V$ in the course of the optimization process. Each element of the design area whose center of gravity is inside this polygon is considered to be ferromagnetic material, the others are considered to be air. That way, we can avoid problems such as deformation of the fixed parts of the motor, i.e. magnets or the air gap, or self-intersections of the mesh.

\begin{figure}
	\begin{minipage}{7cm}
		\begin{tabular}{c} 
			\hspace{-20mm}\includegraphics[scale=0.4, trim=200 0 200 0, clip=true]{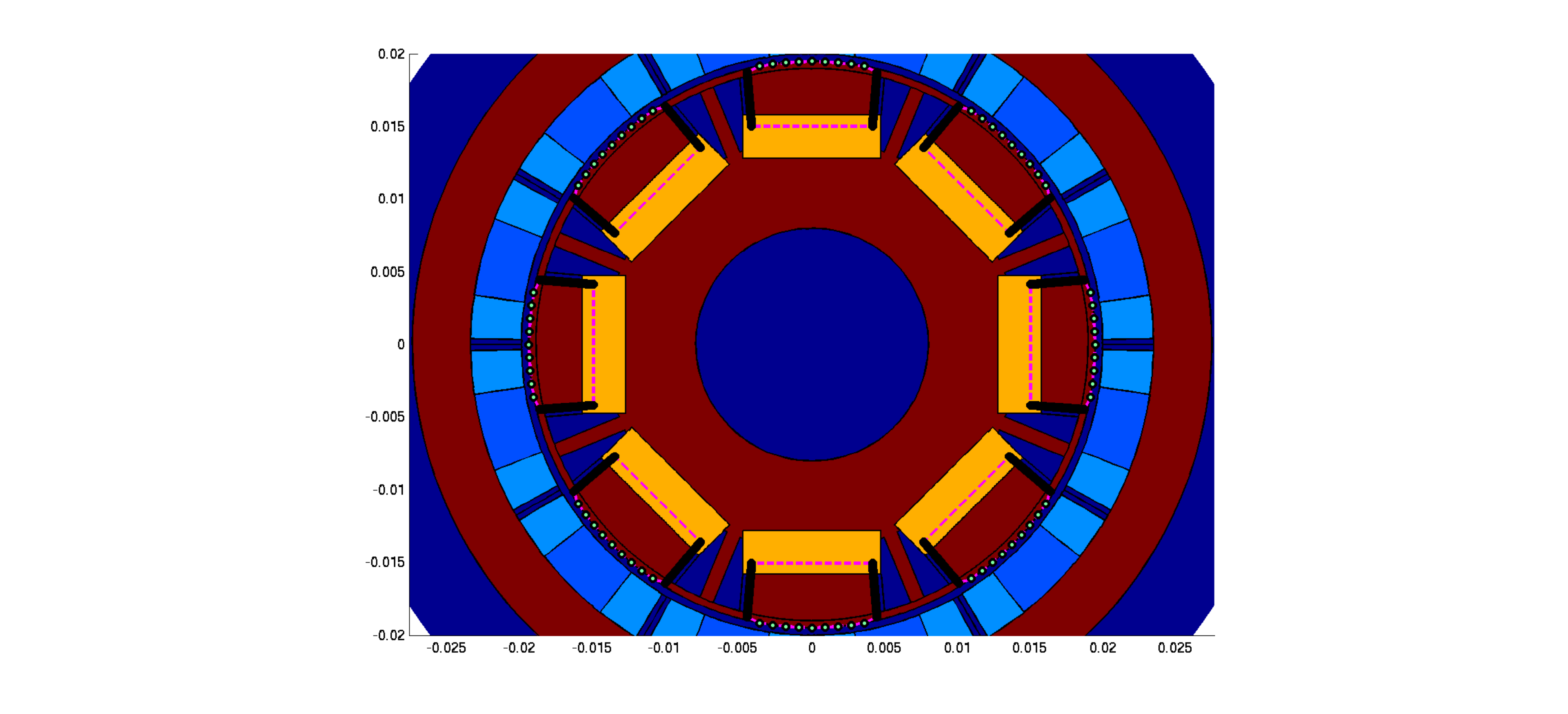} \\
			(a)
		\end{tabular}
	\end{minipage}
    \hspace{1.2cm}
	\begin{minipage}{7cm}
		\begin{tabular}{c} 
			\includegraphics[scale=0.36, trim=200 0 200 0, clip=true]{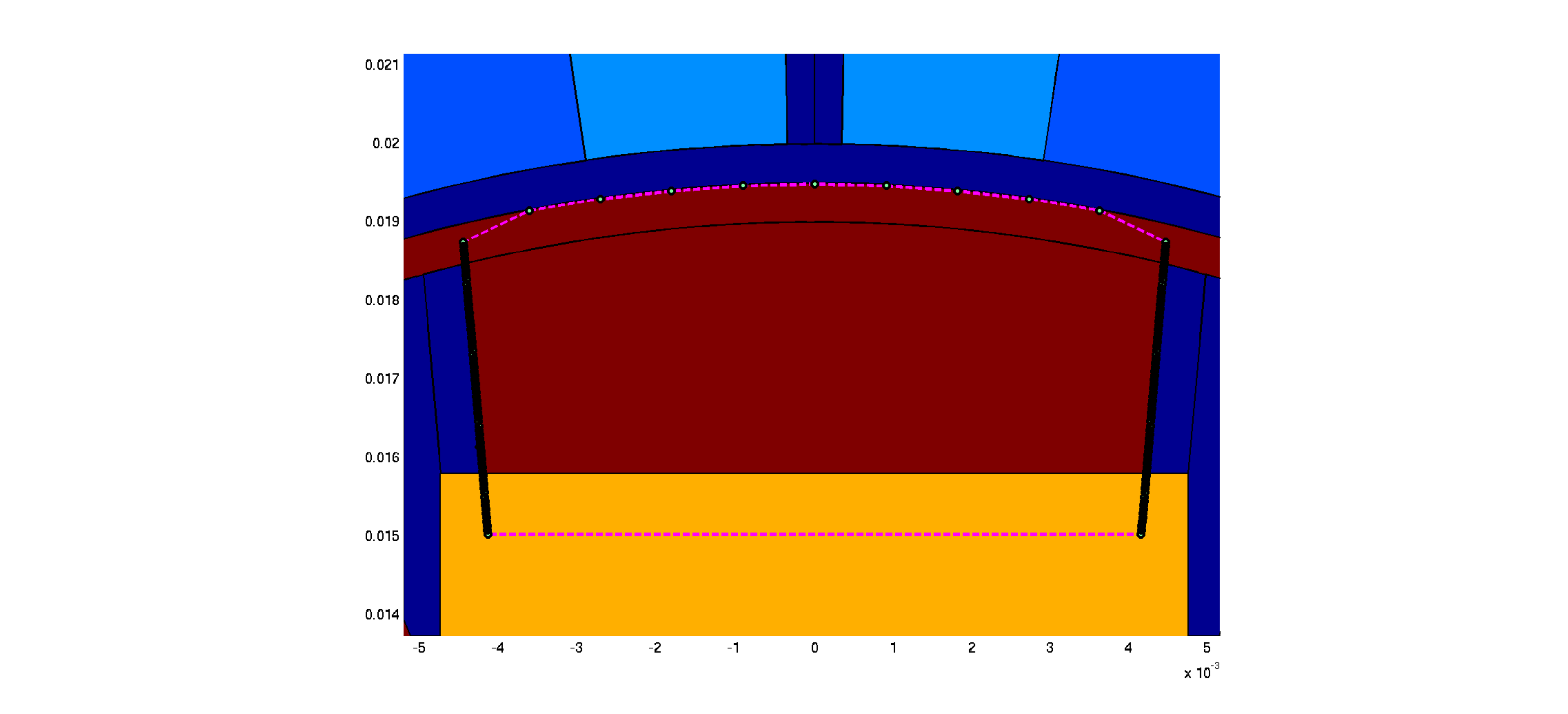}\\
			(b)
		\end{tabular} 
	\end{minipage}
	\caption{(a) Interface points around eight design areas of the electric motor. (b) Zoom on the upper design area: Fictitious interface polygon consists of 151 points (71 on the left, 71 on the right and 9 on top) }
	\label{fig:initInterface}
\end{figure}

%%%%%%%%%%%%%%%%%%%%%%%%%%%%%%%
\subsubsection{Descent Direction}
%%%%%%%%%%%%%%%%%%%%%%%%%%%%%%%
In order to get a descent in the cost functional, we compute the velocity field as follows. We choose a symmetric and positive definite bilinear form
\begin{equation*}
	b: H^1_0(D_{\text{rot}}) \times H^1_0(D_{\text{rot}})  \rightarrow \mathbb{R}
\end{equation*}
defined on the subdomain $D_{\text{rot}}$ of $D$ representing the rotor 
and compute $V$ as the solution of the variational problem: 
\begin{align} 
	&\mbox{Find } V \in P_h \mbox{ such that } b(V,W) = -dJ(\Omega, W) \mbox{ for all } \, W \in P_h,	\label{eq:bvp_V}
\end{align}
where $P_h\subset H^1_0(D_{\text{rot}})$ is a finite dimensional subspace. Outside $D_{\text{rot}}$ we extend $V$ by zero. Note that, by this choice, the condition $V=0$ on $\Gamma_0$, which is assumed in Section \ref{sec:shapeDerivativeCostFunc}, is satisfied.
The obtained descent directions $V \in P_h$ will also be in $W^{1,\infty}(D,\R^2)$ and, 
consequently, they are admissible vector fields defining a flow $T_t^V$.
The solution $V$ computed this way is a descent direction for the cost functional since
\begin{equation*}
	dJ(\Omega, V) = -b(V, V) \leq 0.
\end{equation*}
For our numerical experiments, we choose the bilinear form
\begin{align} \label{auxiliaryBVP}
	\begin{aligned}
	b:H^1_0(D_{\text{rot}}) \times H^1_0(D_{\text{rot}}) &\rightarrow \mathbb R \\
	b(V,W) &= \int_{D_{\text{rot}}} \left(\alpha \, \mbox D V : \mbox D W + V \cdot W \right) \, \mbox d x.
	\end{aligned}
\end{align}
Here, the penalization function $\alpha \in L^{\infty}(D_{\text{rot}})$ is chosen as 
\begin{equation*}
 	\alpha(x) = \left\lbrace \begin{array}{ll}
							1			&	x \in \Omega \\
							10			& 	x \in \Omega^{\varepsilon} \setminus \Omega \\
							10^{2}	&	\mbox{else,}
 	                    \end{array}
 	                    \right. 	
\end{equation*}
where $\Omega^{\varepsilon} = \lbrace x \in D_{\text{rot}}: \mbox{dist}(x, \Omega) \leq \varepsilon \rbrace$ for some small $\varepsilon >0$.
With this choice of $\alpha$, we ensure that the resulting velocity field $V$ is small outside the design region $\Omr$.

For all numerical simulations, we used piecewise linear finite elements on a triangular grid with 44810 degrees of freedom and 89454 elements where we chose a particularly fine discretization in the design regions $\Omr$ (53488 design elements). The nonlinear state equation \eqref{state} is solved by Newton's method. All arising linear systems of finite element equations are solved using the parallel direct solver PARDISO \cite{pardiso}.

\subsubsection{Updating the interface}
For updating the interface, we perform a backtracking (line search) algorithm:
Once a descent direction $V$ is computed, we move all interface points a step size $\tau = \tau_{init}$ in the direction given by $V$ and evaluate the cost function for the updated geometry. If the cost value has not decreased, the step size $\tau$ is halved and the cost function is evaluated for the new, updated geometry. We repeat this step until a decrease of the cost function has been achieved.
When the step size becomes too small such that no element switches its state, the algorithm is stopped.

\subsection{Numerical Results}
The procedure is summarized in Algorithm \ref{algo:shapeOptimization}:
\begin{algorithm} \label{algo:shapeOptimization} Set converged = false \\
	While !converged
	\begin{enumerate}
		\item Compute state $u$ using \eqref{state} and adjoint state $p$ using \eqref{adj}.
		\item Compute shape derivative $dJ(\Omega; V)$ from \eqref{volform}.
		\item Compute descent direction $V$ using \eqref{eq:bvp_V}.
		\item Find step width $\tau$ that yields a decrease in the cost function using backtracking.
		\item If decrease in the cost function could be achieved: 
		\item[]\hspace{10mm} Update interface and go to step 1.\\
		else:
		\item[]\hspace{10mm} converged = true.
	\end{enumerate}
\end{algorithm}
The final design after 35 iterations of Algorithm \ref{algo:shapeOptimization} can be seen in Fig. \ref{fig:finalResult}. The cost function is reduced from $1.3033*10^{-3}$ to $0.94997*10^{-3}$, i.e., by about 27\%. The radial component of the magnetic field on the circle $\Gamma_0$ in the air gap for the initial (blue), desired (green) and final design (red) can be seen in Fig. \ref{fig:Br_J}. The optimization process took about 26 minutes using a single core on a laptop.

\begin{figure} 
	\begin{minipage}{7cm}
		\begin{tabular}{c} 
			\hspace{-20mm}\includegraphics[scale=0.4, trim=200 0 200 0, clip=true]{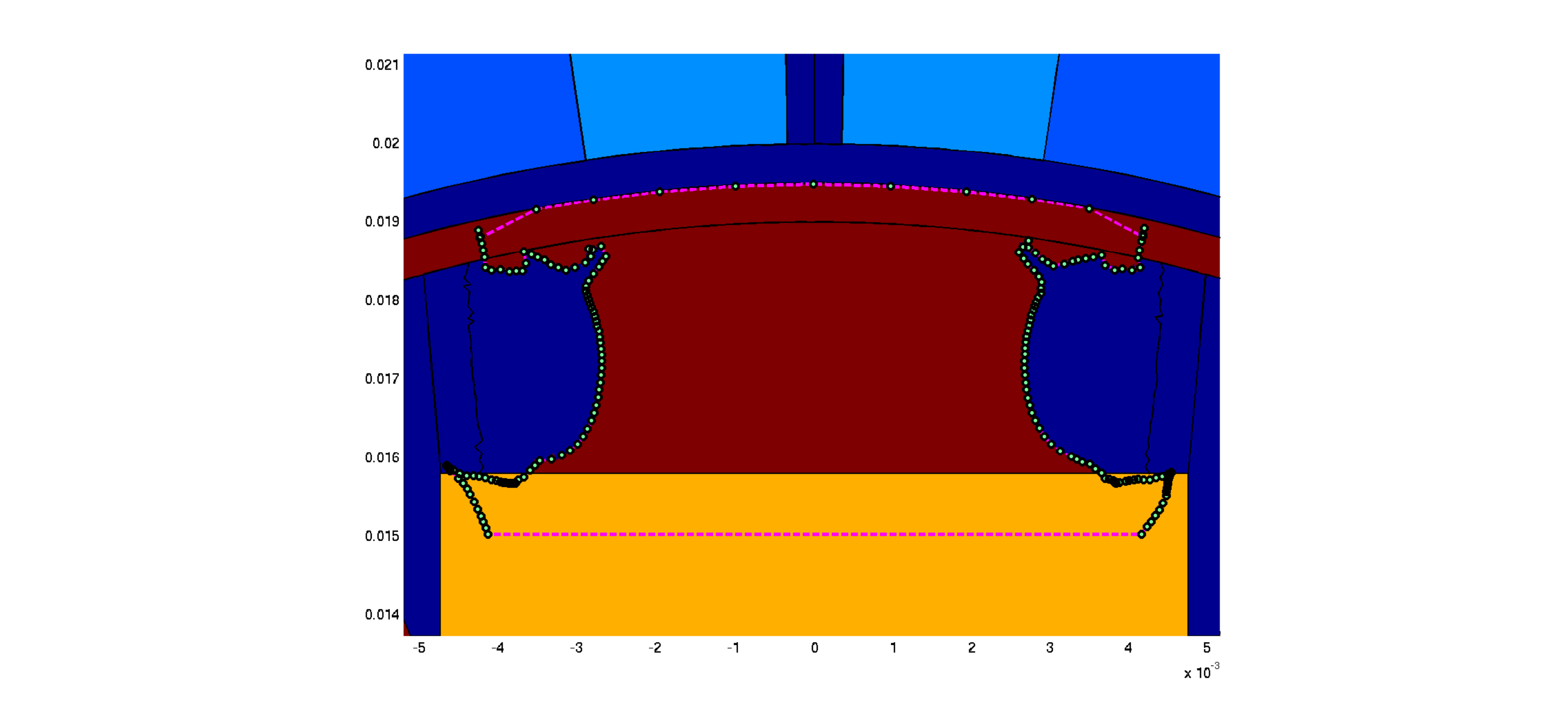} \\
			(a)
		\end{tabular}
	\end{minipage}
    \hspace{1.2cm}
	\begin{minipage}{7cm}
		\begin{tabular}{c} 
			\includegraphics[scale=0.43, trim=200 0 200 0, clip=true]{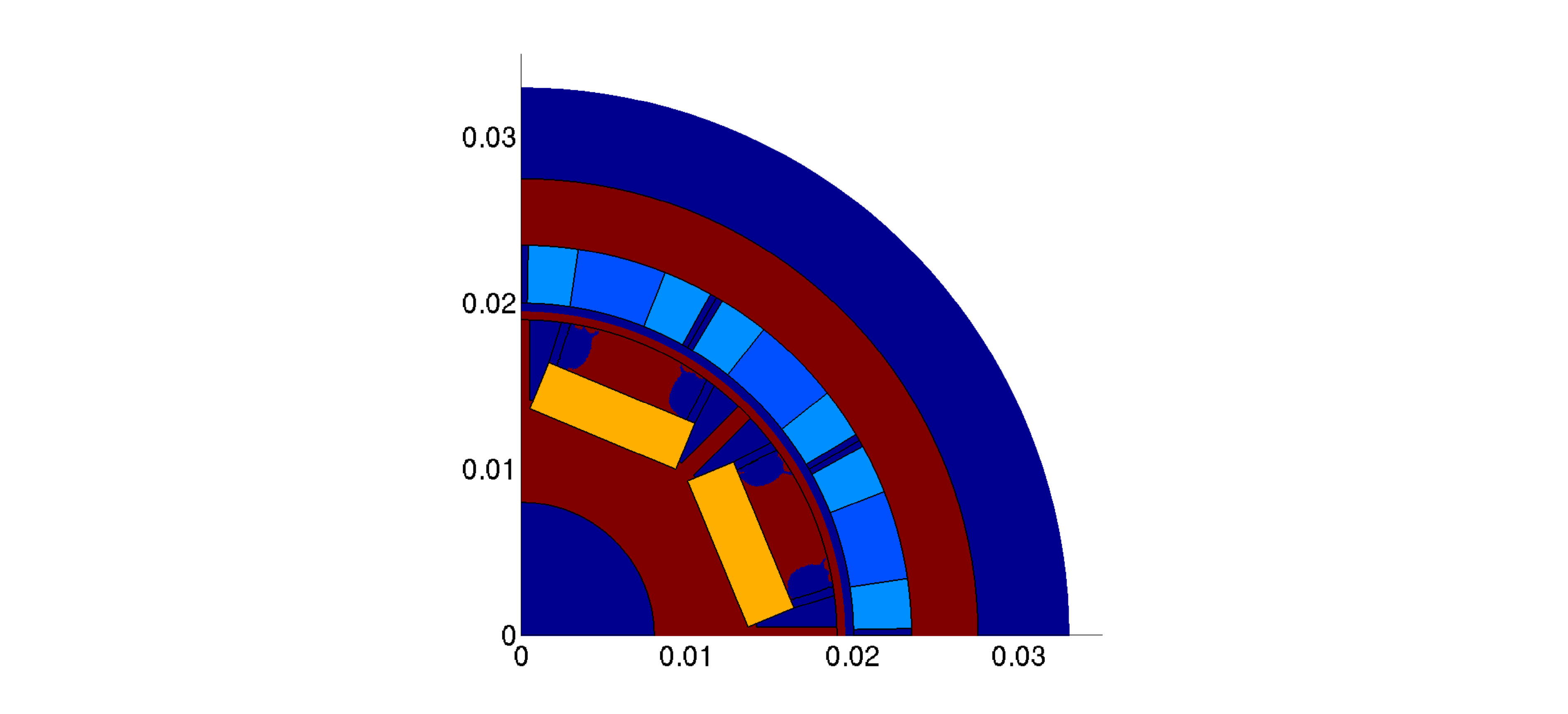}\\
			(b)
		\end{tabular} 
	\end{minipage}
	\caption{(a) Final design of one component $\Omega_k$ after 35 iterations of Algorithm \ref{algo:shapeOptimization}. (b) Upper-right quarter of the optimized motor.}
	\label{fig:finalResult}
\end{figure}

\begin{figure} 
	\begin{minipage}{7cm}
		\begin{tabular}{c} 
			\hspace{-5mm}\includegraphics[scale=0.6]{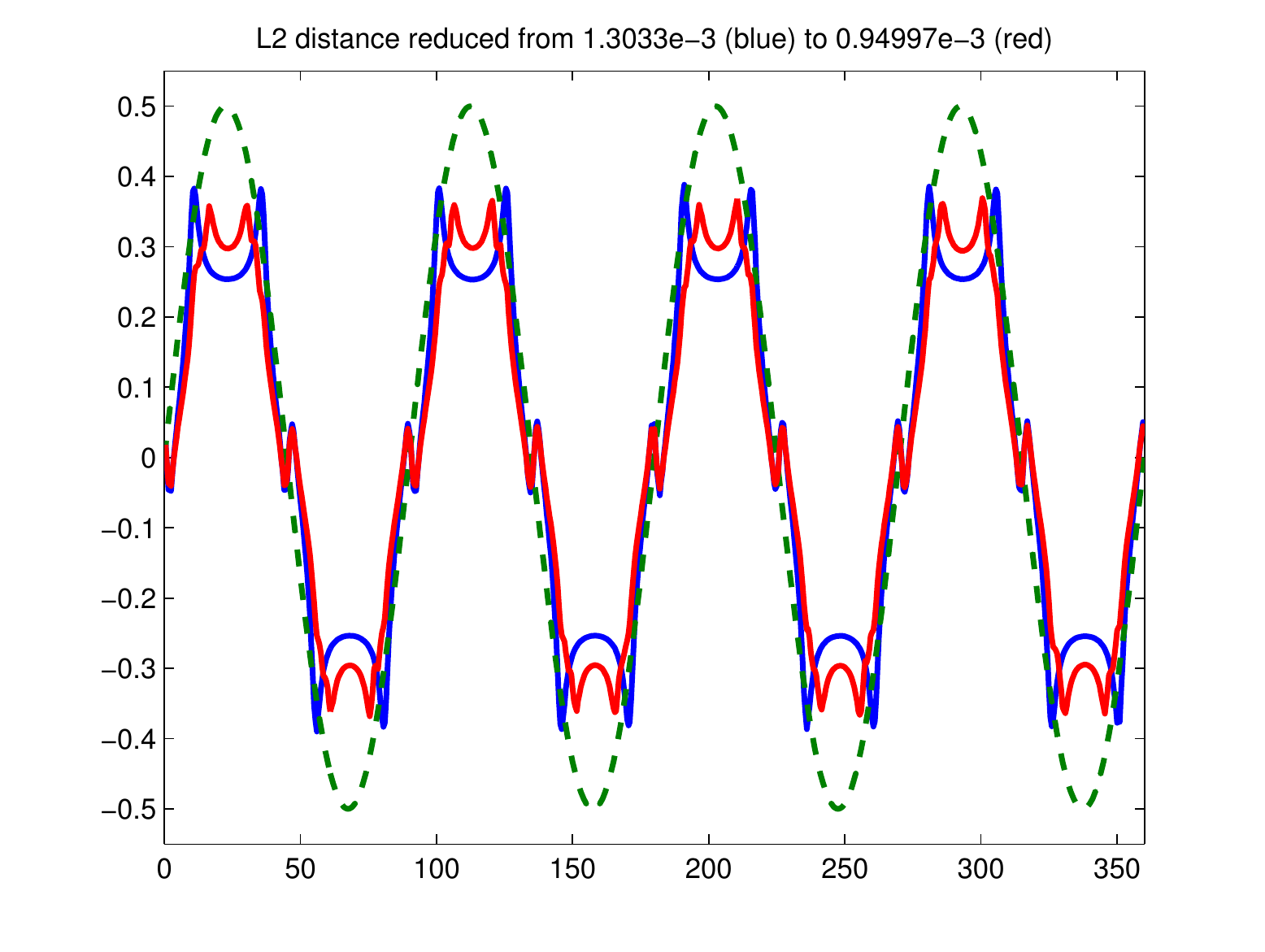} \\
			(a)
		\end{tabular} 
	\end{minipage}
	\hspace{1.2cm}
	\begin{minipage}{7cm}
		\begin{tabular}{c} 
			\includegraphics[scale=0.58]{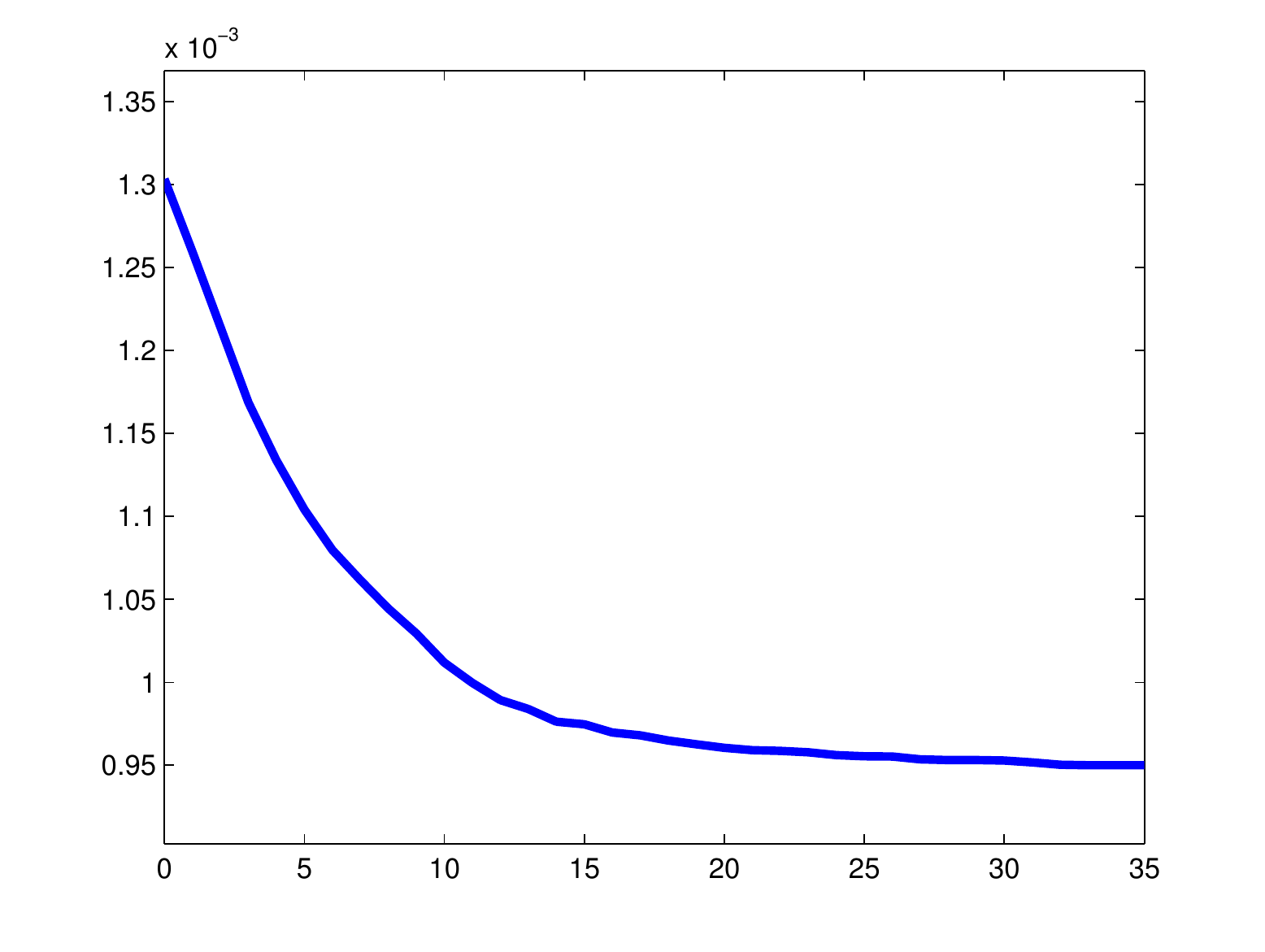} \\
			(b)
		\end{tabular} 
	\end{minipage}
	\caption{(a) Radial component of the magnetic flux density $B$ along curve $\Gamma_0$ in air gap: initial design (blue), desired sine curve $B_d$ (green), final design (red). (b) Decrease of cost functional in the course of the optimization process.}
	\label{fig:Br_J}
\end{figure}

\section{Conclusion}

In this paper, we have performed the rigorous analysis of the shape sensitivity analysis of a subregion $\Omega$ of the rotor of an electric motor in order to match a certain rotation pattern. The shape derivative of the cost functional was computed efficiently using a shape-Lagrangian method for nonlinear partial differential equation constraints which allows to bypass the computation of the material derivative of the state.
The implementation of the obtained shape derivative in a numerical algorithm provides an interesting shape which allows us to improve the rotation pattern. In the numerical experiment presented in this work, we chose a rather simple way of updating the interface: We just switched the state of single elements of the finite element mesh as discrete entities. For a more accurate resolution of the interface, one may employ a nonstandard finite element method like the extended Finite Element Method (XFEM) \cite{MoesDolbowBelytschko1999} or the immersed FEM \cite{LiLinWu2003}, or a discontinuous Galerkin 
approach based on Nitsche's idea, see \cite{HansboHansbo2002}. These approaches make it possible to represent an interface that is not aligned with the underlying FE discretization without loss of accuracy. An alternative way to achieve this would be to locally modify the finite element basis in a parametric way as it is done in \cite{FreiRichter2014}.\\

\noindent\textbf{Acknowledgments.} Antoine Laurain and Houcine Meftahi acknowledge financial support by the DFG Research Center Matheon ``Mathematics for key technologies'' through the MATHEON-Project C37 ``Shape/Topology optimization methods for inverse problems''. Peter Gangl and Ulrich Langer gratefully acknowledge the Austrian Science Fund (FWF) for the financial support of their work via the Doctoral Program DK W1214 (project DK4) on Computational Mathematics. They also thank the Linz Center of Mechatronics (LCM), which is a part of the COMET K2 program of the Austrian Government, for supporting their work on topology and shape optimization of electrical machines.
\bibliographystyle{abbrv}
\bibliography{biblio2}

\end{document}